\documentclass[12pt]{article}
\usepackage{e-jc}

\usepackage{amsthm,amsmath,amssymb}

\usepackage{graphicx}

\usepackage[colorlinks=true,citecolor=black,linkcolor=black,urlcolor=blue]{hyperref}

\theoremstyle{plain}
\newtheorem{theorem}{Theorem}[section]
\newtheorem{lemma}[theorem]{Lemma}

\theoremstyle{definition}

\theoremstyle{remark}
\newtheorem{remark}[theorem]{Remark}

\numberwithin{equation}{section}
\numberwithin{figure}{section}
\numberwithin{table}{section}

\newcommand{\M}{\operatorname{M}}
\newcommand{\E}{\operatorname{E}}
\newcommand{\Od}{\operatorname{O}}
\newcommand{\wt}{\operatorname{wt}}

\title{Enumeration of tilings of quartered Aztec rectangles}

\author{Tri Lai\\
\small Department of Mathematics\\[-0.8ex]
\small Indiana University\\[-0.8ex]
\small Bloomington, IN 47405\\
\small\tt tmlai@indiana.edu
}

\date{\small Mathematics Subject Classifications: 05A15, 05C70, 05E99}

\begin{document}

\maketitle


\begin{abstract}
We generalize a theorem of  W. Jockusch and J. Propp on quartered Aztec diamonds by enumerating the tilings of quartered Aztec rectangles. We use subgraph replacement method to transform the dual graph of a quartered Aztec rectangle to the dual graph of a quartered lozenge hexagon, and then use Lindstr\"{o}m-Gessel-Viennot methodology to find the number of tilings of a quartered lozenge hexagon.

\end{abstract}

\section{Introduction}

A lattice divides the plane into fundamental regions. A (lattice) \textit{region} is a finite connected union of fundamental regions of that lattice. A \textit{tile} is the union of any two fundamental regions sharing an edge. A \textit{tiling} of the region $R$ is a covering of $R$ by tiles with no gaps or overlaps. Denote by $\M(R)$ the number of tilings of the region $R$.

In general, the tiles of a region  $R$ can carry weights. The \textit{weight of a tiling} is defined to be the product of the weights of all constituent tiles.  The operation $\M(R)$ is now defined to be the sum of the weights of all tilings in $R$, and is called the \textit{tiling generating function} of $R$. If $R$ does not have any tiling, we let $\M(R):=0$.

 The \textit{Aztec diamond} of order $n$ is defined to be the union of all the unit squares with integral corners $(x,y)$ satisfying $|x|+|y|\leq n+1$. The Aztec diamond  of order $8$ is shown in Figure \ref{QA}. We denote by $\mathcal{AD}_{n}$ the Aztec diamond of order $n$. It has been shown that the number of tilings of $AD_n$ is $2^{\frac{n(n+1)}{2}}$ (see \cite{Elkies}).

We are interested in three related families of regions first introduced by Jockusch and Propp \cite{JP} as follows. Divide the Aztec diamond of order $n$ into two congruent parts by a zigzag cut with 2-unit steps. By superimposing two such zigzag cuts that pass the center of the Aztec diamond we partition the region into four parts, called \textit{quartered Aztec diamonds}. Up to symmetry, there are essentially two different ways we can superimpose the two cuts. For one of them, we obtained a fourfold rotational symmetric pattern, and four resulting parts are congruent. Denote by $R(n)$ these quartered Aztec diamonds (see Figure \ref{QA}(a)). For the other, the obtained pattern has Klein 4-group reflection symmetry and there are two different kinds of quartered Aztec diamonds (see Figure \ref{QA} (b)); they are called \textit{abutting} and \textit{non-abutting} quartered Aztec diamonds. Denote by $K_{a}(n)$ and $K_{na}(n)$ the abutting and non-abutting quartered Aztec diamonds of order $n$, respectively.

\begin{figure}\centering
\includegraphics[width=13cm]{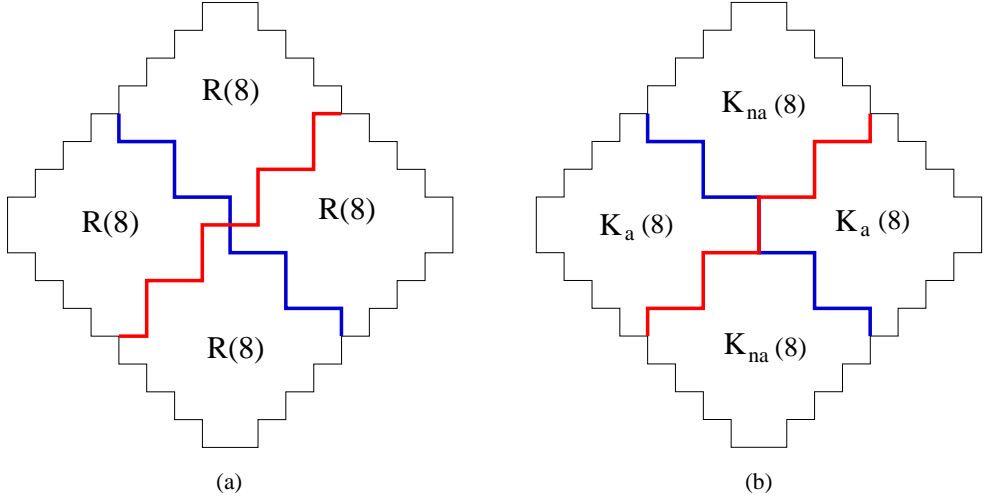}
\caption{Three kinds of quartered Aztec diamonds of order 8. The figure first appeared in \cite{Tri2}.}
\label{QA}
\end{figure}

\begin{figure}\centering
\begin{picture}(0,0)%
\includegraphics{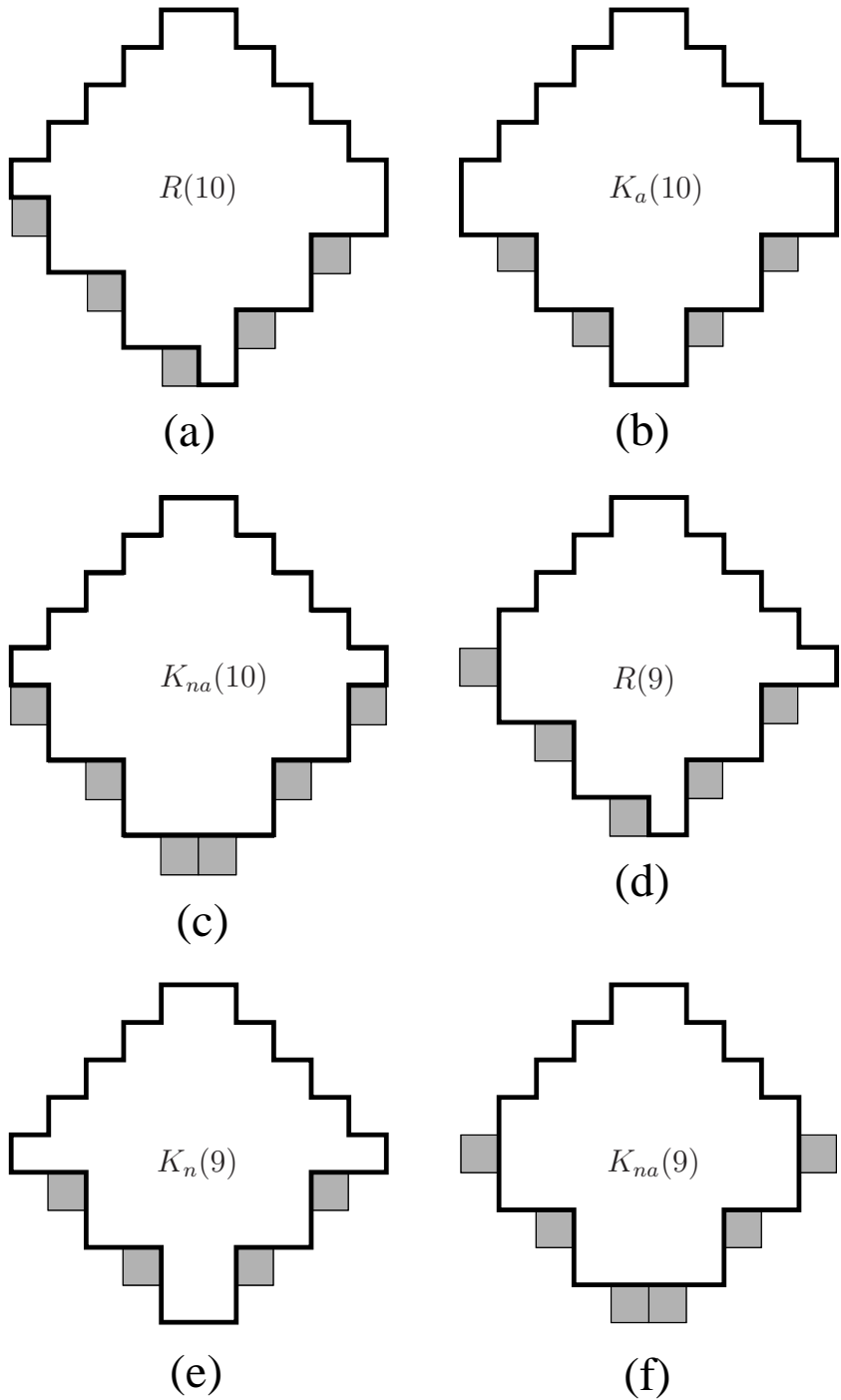}%
\end{picture}%
%
%
\setlength{\unitlength}{3947sp}%
\begingroup\makeatletter\ifx\SetFigFont\undefined%
\gdef\SetFigFont#1#2#3#4#5{%
  \reset@font\fontsize{#1}{#2pt}%
  \fontfamily{#3}\fontseries{#4}\fontshape{#5}%
  \selectfont}%
\fi\endgroup%
\begin{picture}(5237,8813)(694,-8663)
\put(1655,-1051){\makebox(0,0)[lb]{\smash{{\SetFigFont{12}{14.4}{\rmdefault}{\mddefault}{\updefault}{$R(10)$}%
}}}}
\put(4489,-1051){\makebox(0,0)[lb]{\smash{{\SetFigFont{12}{14.4}{\rmdefault}{\mddefault}{\updefault}{$K_a(10)$}%
}}}}
\put(1662,-4130){\makebox(0,0)[lb]{\smash{{\SetFigFont{12}{14.4}{\rmdefault}{\mddefault}{\updefault}{$K_{na}(10)$}%
}}}}
\put(4510,-4142){\makebox(0,0)[lb]{\smash{{\SetFigFont{12}{14.4}{\rmdefault}{\mddefault}{\updefault}{$R(9)$}%
}}}}
\put(1644,-7188){\makebox(0,0)[lb]{\smash{{\SetFigFont{12}{14.4}{\rmdefault}{\mddefault}{\updefault}{$K_n(9)$}%
}}}}
\put(4479,-7188){\makebox(0,0)[lb]{\smash{{\SetFigFont{12}{14.4}{\rmdefault}{\mddefault}{\updefault}{$K_{na}(9)$}%
}}}}
\end{picture}
\caption{Obtaining quartered Aztec diamonds from Aztec diamonds and trimmed Aztec diamonds.}
\label{QAztec}
\end{figure}

For $a_1<a_2<\dotsc<a_n$, we define two functions by setting
\begin{equation}\label{E}
\E(a_1,a_2,\dotsc,a_n)=\frac{2^{n^2}}{0!2!4!\dotsc(2n-2)!}\prod_{1\leq i<j\leq n}(a_j-a_i)\prod_{1\leq i<j\leq n}(a_i+a_j-1),
\end{equation}
\begin{equation}\label{O}
\Od(a_1,a_2,\dotsc,a_n)=\frac{2^{n^2}}{1!3!5!\dotsc(2n-1)!}\prod_{1\leq i<j \leq n}(a_j-a_i)\prod_{1\leq i\leq j \leq n}(a_i+a_j-1).
\end{equation}
Hereafter, the empty products (like $\prod_{1\leq i<j\leq n}(a_j-a_i)$ for $n=1$) equal 1 by convention.
The above two functions have a special connection to the weighted sum of \textit{antisymmetric monotone triangles} (see \cite{JP}).

The number of tilings of a quartered Aztec diamond is given by the theorem stated below.
\begin{theorem}[Jockusch and Propp \cite{JP}]\label{oldquartered}
For any positive integer $n$
\begin{equation}\label{oldmain1}
\M(R(4n+1))=\M(R(4n+2))=0,
\end{equation}
\begin{equation}\label{oldmain2}
\M(R(4n))=2^n\M(R(4n-1))=\E(2,4,\dotsc,2n),
\end{equation}
\begin{equation}\label{oldmain3}
\M(K_{a}(4n-2))=\M(K_a(4n))=\E(1,3,5,\dotsc,2n-1),
\end{equation}
\begin{equation}\label{oldmain4}
\M(K_{a}(4n-1))=\M(K_a(4n+1))=2^{-n}\E(1,3,5,\dotsc,2n-1),
\end{equation}
\begin{equation}\label{oldmain5}
\M(K_{na}(4n))=\M(K_{na}(4n+2))=\Od(2,4,\dotsc,2n),
\end{equation}
\begin{equation}\label{oldmain6}
\M(K_{na}(4n-3))=\M(K_{na}(4n-1))=2^{-n}\E(1,3,5,\dotsc,2n-1).
\end{equation}
\end{theorem}

We notice that
\begin{equation}
\E(2,4,\dotsc,2n)=\Od(1,3,5,\dotsc,2n-1)=2^{n(3n-1)/2}\prod_{1\leq i<j\leq n}\frac{2i+2j-1}{i+j-1},
\end{equation}
and that the author presented a simple proof for Theorem \ref{oldquartered} in \cite{Tri2}.

\medskip

We define a \textit{trimmed Aztec diamond} of order $n$ to be the region obtained from an Aztec diamond of order $n$ by removing the squares running along the northwestern and northeastern side, denoted by $\mathcal{TA}_n$.

\begin{figure}\centering
\includegraphics[width=12cm]{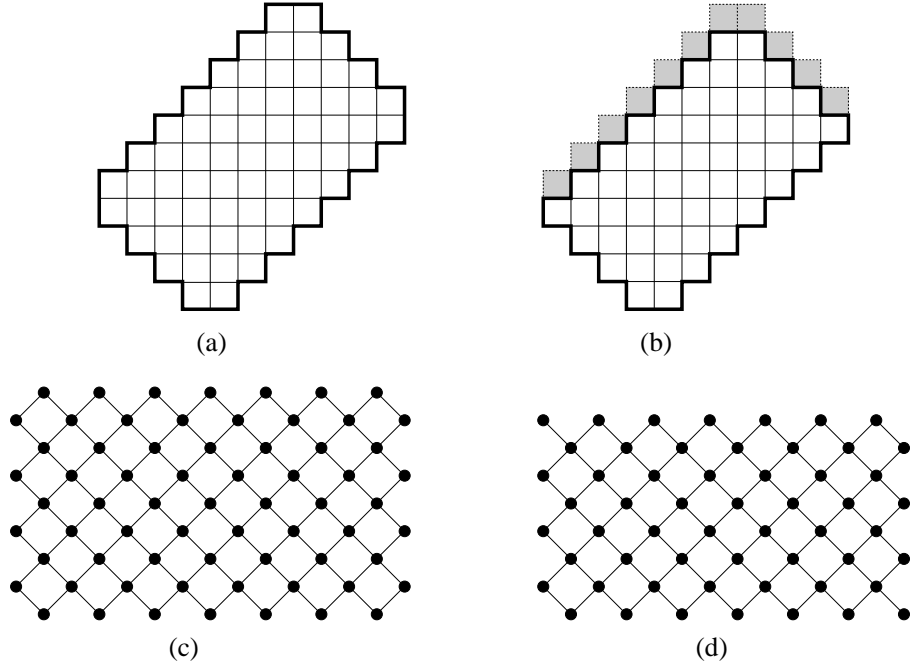}
\caption{The Aztec rectangle $AR_{4,7}$ (a), the trimmed Aztec rectangle $TR_{4,7}$ (b), the dual graph (rotated $45^0$ clockwise) of $AR_{4,7}$ (c), and the dual graph of $TR_{4,7}$ (d). }
\label{ARregion}
\end{figure}

\begin{figure}\centering
\resizebox{!}{15cm}{
\begin{picture}(0,0)%
\includegraphics{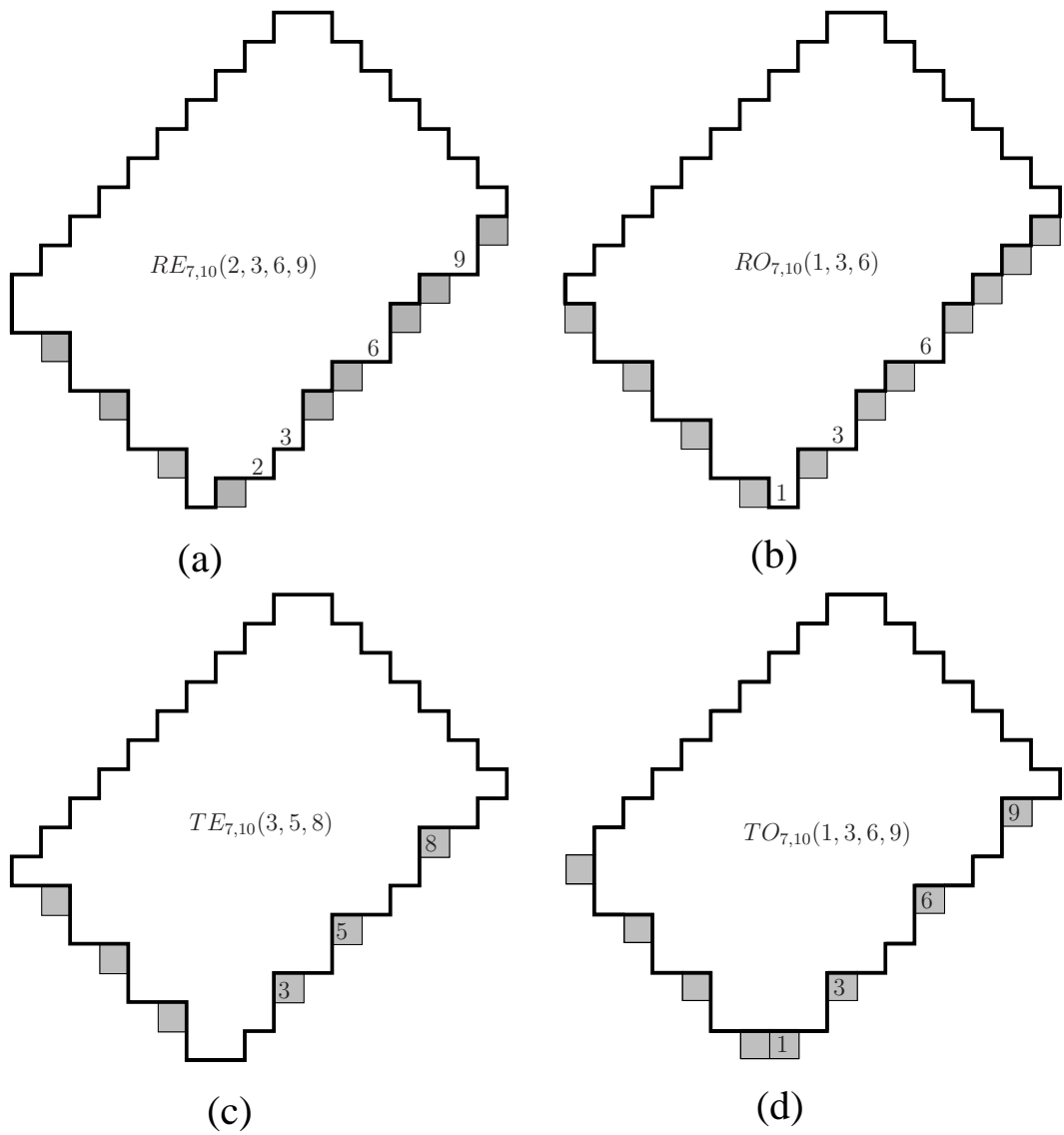}%
\end{picture}%
%
%
\setlength{\unitlength}{3947sp}%
\begingroup\makeatletter\ifx\SetFigFont\undefined%
\gdef\SetFigFont#1#2#3#4#5{%
  \reset@font\fontsize{#1}{#2pt}%
  \fontfamily{#3}\fontseries{#4}\fontshape{#5}%
  \selectfont}%
\fi\endgroup%
\begin{picture}(8529,9144)(1413,-8762)
\put(3391,-3377){\makebox(0,0)[lb]{\smash{{\SetFigFont{12}{14.4}{\rmdefault}{\mddefault}{\updefault}{$2$}%
}}}}
\put(3623,-3137){\makebox(0,0)[lb]{\smash{{\SetFigFont{12}{14.4}{\rmdefault}{\mddefault}{\updefault}{$3$}%
}}}}
\put(4332,-2417){\makebox(0,0)[lb]{\smash{{\SetFigFont{12}{14.4}{\rmdefault}{\mddefault}{\updefault}{$6$}%
}}}}
\put(5033,-1697){\makebox(0,0)[lb]{\smash{{\SetFigFont{12}{14.4}{\rmdefault}{\mddefault}{\updefault}{$9$}%
}}}}
\put(7644,-3594){\makebox(0,0)[lb]{\smash{{\SetFigFont{12}{14.4}{\rmdefault}{\mddefault}{\updefault}{$1$}%
}}}}
\put(8101,-3121){\makebox(0,0)[lb]{\smash{{\SetFigFont{12}{14.4}{\rmdefault}{\mddefault}{\updefault}{$3$}%
}}}}
\put(8814,-2401){\makebox(0,0)[lb]{\smash{{\SetFigFont{12}{14.4}{\rmdefault}{\mddefault}{\updefault}{$6$}%
}}}}
\put(2566,-1741){\makebox(0,0)[lb]{\smash{{\SetFigFont{12}{14.4}{\rmdefault}{\mddefault}{\updefault}{$RE_{7,10}(2,3,6,9)$}%
}}}}
\put(7306,-1719){\makebox(0,0)[lb]{\smash{{\SetFigFont{12}{14.4}{\rmdefault}{\mddefault}{\updefault}{$RO_{7,10}(1,3,6)$}%
}}}}
\put(3609,-7629){\makebox(0,0)[lb]{\smash{{\SetFigFont{12}{14.4}{\rmdefault}{\mddefault}{\updefault}{$3$}%
}}}}
\put(4074,-7149){\makebox(0,0)[lb]{\smash{{\SetFigFont{12}{14.4}{\rmdefault}{\mddefault}{\updefault}{$5$}%
}}}}
\put(4801,-6436){\makebox(0,0)[lb]{\smash{{\SetFigFont{12}{14.4}{\rmdefault}{\mddefault}{\updefault}{$8$}%
}}}}
\put(7651,-8071){\makebox(0,0)[lb]{\smash{{\SetFigFont{12}{14.4}{\rmdefault}{\mddefault}{\updefault}{$1$}%
}}}}
\put(8110,-7609){\makebox(0,0)[lb]{\smash{{\SetFigFont{12}{14.4}{\rmdefault}{\mddefault}{\updefault}{$3$}%
}}}}
\put(8821,-6901){\makebox(0,0)[lb]{\smash{{\SetFigFont{12}{14.4}{\rmdefault}{\mddefault}{\updefault}{$6$}%
}}}}
\put(9529,-6190){\makebox(0,0)[lb]{\smash{{\SetFigFont{12}{14.4}{\rmdefault}{\mddefault}{\updefault}{$9$}%
}}}}
\put(2887,-6268){\makebox(0,0)[lb]{\smash{{\SetFigFont{12}{14.4}{\rmdefault}{\mddefault}{\updefault}{$TE_{7,10}(3,5,8)$}%
}}}}
\put(7384,-6349){\makebox(0,0)[lb]{\smash{{\SetFigFont{12}{14.4}{\rmdefault}{\mddefault}{\updefault}{$TO_{7,10}(1,3,6,9)$}%
}}}}
\end{picture}}
\caption{Obtaining the quartered Aztec rectangles from Aztec rectangles and trimmed Aztec rectangles.}
\label{QAztec2}
\end{figure}

Label the squares on the southwestern and southeastern sides of $\mathcal{AD}_n$ and $\mathcal{TA}_n$ by $1,2,\dots,n$ from bottom to top. One readily sees that the region $R(2k)$ (resp., $R(2k-1)$) is obtained from $\mathcal{AD}_k$ (reps., $\mathcal{TA}_{k}$) by removing odd squares on its southwestern side, and even squares on its southeastern side. Similarly, the region $K_{a}(2k)$ (resp., $K_{a}(2k-1)$) is obtained from the region $\mathcal{AD}_k$ (reps., $\mathcal{TA}_{k}$) by removing even squares from both southwestern and southeastern sides; the region $K_{na}(2k)$ (resp., $K_{na}(2k-1)$) is obtained from the region $\mathcal{AD}_k$ (reps., $\mathcal{TA}_{k}$) by removing odd squares from the two sides (see Figure \ref{QAztec} for examples; the quartered Aztec diamonds are the ones restricted by the bold contours; the shaded squares indicate the ones removed).

Besides Aztec diamonds, we are interested in a similar families of regions called \textit{Aztec rectangles}.  See Figures \ref{ARregion}(a)  and (c) for an example of the Aztec rectangle of order $(4,7)$ and its \textit{dual graph}, i.e. the graph whose vertices are the unit squares of the region and the edges connect exactly two unit squares sharing a side. Denote by $\mathcal{AR}_{m,n}$ the Aztec rectangle region of order $(m,n)$. We also consider the \textit{trimmed Aztec rectangle} region $\mathcal{TR}_{m,n}$ obtained from $\mathcal{AR}_{m,n}$ by removing squares running along its northwestern and northeastern sides (see Figures \ref{ARregion}(b) and (d) for a trimmed Aztec rectangle and its dual graph). We notice that the regions $\mathcal{AD}_n$ and $\mathcal{TA}_{n}$ are obtained from $\mathcal{AR}_{m,n}$ and $\mathcal{TR}_{m,n}$, respectively, by specializing $m=n$.

Similar to quartered Aztec diamonds, we consider the region obtained from $\mathcal{AR}_{m,n}$ by removing even squares on the southwestern side, and removing  \textit{arbitrarily} $n-\lfloor\frac{m+1}{2}\rfloor$  squares on the southeastern sides. Assume that we are removing all the squares, except for the $a_1$-st, the $a_2$-nd, $\dotsc$, and the $a_{\lfloor\frac{m+1}{2}\rfloor}$-th ones, from the southeastern side, then we denote by $RE_{m,n}(a_1,a_2,\dotsc,a_{\lfloor\frac{m+1}{2}\rfloor})$ the resulting region (see Figure \ref{QAztec2}(a)). Next, we remove all odd squares from the southwestern side of $\mathcal{AR}_{m,n}$, and remove all squares, except for the ones with labels $a_1<a_2<\dotsc<a_{\lfloor\frac{m}{2}\rfloor}$, from the southeastern side. We denote by $RO_{m,n}(a_1,a_2,\dotsc,a_{\lfloor\frac{m}{2}\rfloor})$ the resulting region (see Figure \ref{QAztec2}(b)).

If we remove all even squares on the southwestern side of $\mathcal{TR}_{m,n}$, and also remove the square $a_1<a_2<\dotsc<a_{\lfloor\frac{m}{2}\rfloor}$ from its southeastern side, then we get the region denoted by $TE_{m,n}(a_1,a_2,\dotsc,a_{\lfloor\frac{m}{2}\rfloor})$ (illustrated in Figure \ref{QAztec2}(c)). Repeat process with the odd squares on the southwestern side removed, we get region $TO_{m,n}(a_1,a_2,\dotsc,a_{\lfloor\frac{m+1}{2}\rfloor})$ (shown in Figure \ref{QAztec2}(d)).

We call the four regions in the previous two paragraphs \textit{quartered Aztec rectangles}. Surprisingly, the numbers of tilings of quartered Aztec rectangles are given by simple product formulas involving two functions $\E(...)$ and $\Od(...)$ defined in (\ref{E}) and (\ref{O}).

\begin{theorem}\label{main} For any $1\leq k< n$ and $1\leq a_1<a_2<\dotsc<a_k\leq n$
\begin{equation}\label{main1}
\M(RE_{2k-1,n}(a_1,a_2,\dotsc,a_k))=\M(RE_{2k,n}(a_1,a_2,\dotsc,a_k))=\E(a_1,a_2,\dotsc,a_k),
\end{equation}
\begin{equation}\label{main2}
\M(RO_{2k,n}(a_1,a_2,\dotsc,a_k))=\M(RO_{2k+1,n}(a_1,a_2,\dotsc,a_k))=\Od(a_1,a_2,\dotsc,a_k),
\end{equation}
\begin{equation}\label{main3}
\M(TE_{2k,n}(a_1,a_2,\dotsc,a_{k}))=\M(TE_{2k+1,n}(a_1,a_2,\dotsc,a_{k}))=2^{-k}\Od(a_1,a_2,\dotsc,a_k),
\end{equation}
\begin{equation}\label{main4}
\M(TO_{2k-1,n}(a_1,a_2,\dotsc,a_{k}))=\M(TO_{2k,n}(a_1,a_2,\dotsc,a_{k}))=2^{-k}\E(a_1,a_2,\dotsc,a_k).
\end{equation}
\end{theorem}

\medskip

The structure of Aztec rectangles allows us to have the four variants of quartered Aztec rectangles as follows.

\begin{figure}\centering
\resizebox{!}{12cm}{
\begin{picture}(0,0)%
\includegraphics{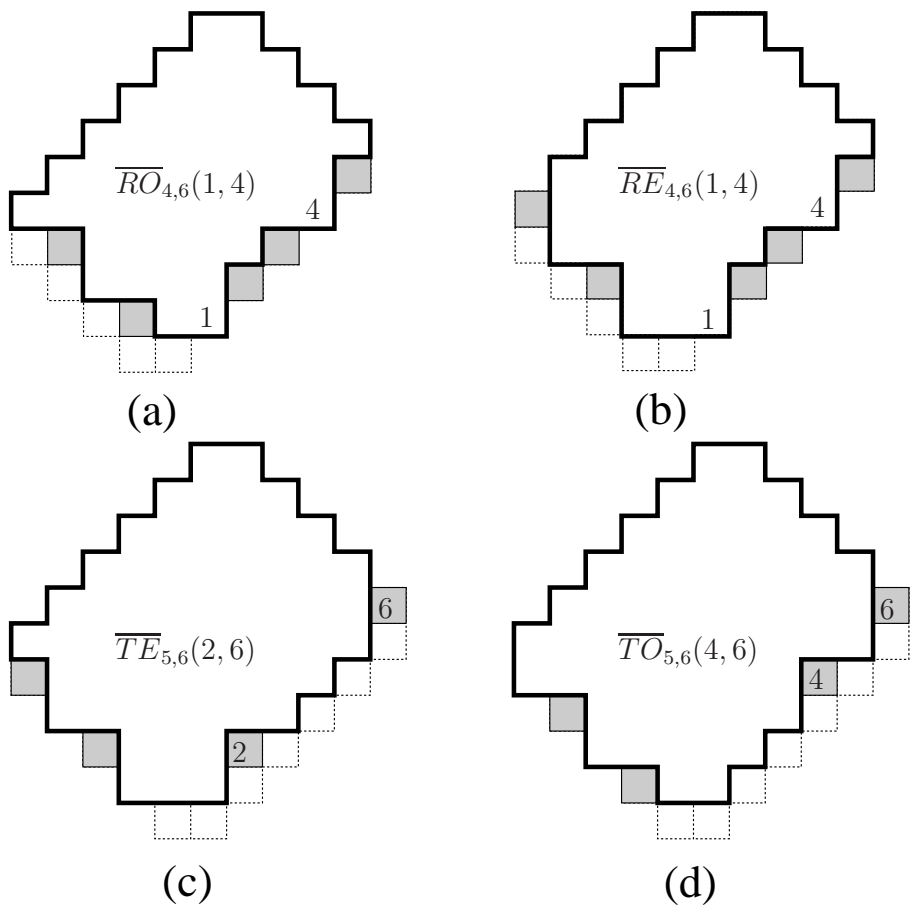}%
\end{picture}%
\setlength{\unitlength}{3947sp}%
\begingroup\makeatletter\ifx\SetFigFont\undefined%
\gdef\SetFigFont#1#2#3#4#5{%
  \reset@font\fontsize{#1}{#2pt}%
  \fontfamily{#3}\fontseries{#4}\fontshape{#5}%
  \selectfont}%
\fi\endgroup%
\begin{picture}(5929,5937)(698,-5784)
\put(5266,-1948){\makebox(0,0)[lb]{\smash{{\SetFigFont{12}{14.4}{\rmdefault}{\mddefault}{\updefault}{$1$}%
}}}}
\put(5986,-1213){\makebox(0,0)[lb]{\smash{{\SetFigFont{12}{14.4}{\rmdefault}{\mddefault}{\updefault}{$4$}%
}}}}
\put(2666,-1222){\makebox(0,0)[lb]{\smash{{\SetFigFont{12}{14.4}{\rmdefault}{\mddefault}{\updefault}{$4$}%
}}}}
\put(1967,-1937){\makebox(0,0)[lb]{\smash{{\SetFigFont{12}{14.4}{\rmdefault}{\mddefault}{\updefault}{$1$}%
}}}}
\put(1418,-1051){\makebox(0,0)[lb]{\smash{{\SetFigFont{12}{14.4}{\rmdefault}{\mddefault}{\updefault}{$\overline{RO}_{4,6}(1,4)$}%
}}}}
\put(4725,-1051){\makebox(0,0)[lb]{\smash{{\SetFigFont{12}{14.4}{\rmdefault}{\mddefault}{\updefault}{$\overline{RE}_{4,6}(1,4)$}%
}}}}
\put(1419,-4122){\makebox(0,0)[lb]{\smash{{\SetFigFont{12}{14.4}{\rmdefault}{\mddefault}{\updefault}{$\overline{TE}_{5,6}(2,6)$}%
}}}}
\put(4726,-4122){\makebox(0,0)[lb]{\smash{{\SetFigFont{12}{14.4}{\rmdefault}{\mddefault}{\updefault}{$\overline{TO}_{5,6}(4,6)$}%
}}}}
\put(2184,-4794){\makebox(0,0)[lb]{\smash{{\SetFigFont{12}{14.4}{\rmdefault}{\mddefault}{\updefault}{$2$}%
}}}}
\put(3144,-3849){\makebox(0,0)[lb]{\smash{{\SetFigFont{12}{14.4}{\rmdefault}{\mddefault}{\updefault}{$6$}%
}}}}
\put(6444,-3856){\makebox(0,0)[lb]{\smash{{\SetFigFont{12}{14.4}{\rmdefault}{\mddefault}{\updefault}{$6$}%
}}}}
\put(5979,-4314){\makebox(0,0)[lb]{\smash{{\SetFigFont{12}{14.4}{\rmdefault}{\mddefault}{\updefault}{$4$}%
}}}}
\end{picture}}
\caption{The four variants of quartered Aztec rectangles.}
\label{QAztec3}
\end{figure}

Start with the Aztec rectangle $\mathcal{AR}_{m,n}$. We first, remove all squares along the southwestern side of the region, and remove also the bottommost square of the resulting region (see the squares with dotted sides in Figure \ref{QAztec3}(a)). We get a region $R$. We also label the squares on the southwestern and southeastern side of $R$ by positive integers from bottom to top. Remove also the even squares on the southwestern side of $R$, and remove all squares, except for the ones with labels $a_1< a_2<\dotsc<a_{\lfloor\frac{m}{2}\rfloor}$, from the southeastern side of $R$. We get a region denoted by $\overline{RE}_{m,n}(a_1,a_2,\dotsc, a_{\lfloor\frac{m}{2}\rfloor})$ (see Figure \ref{QAztec3}(a)). Repeat the process, however, we remove all even squares from the southwestern side of $R$ (as opposed to odd squares), and remove again all the squares, except for the ones with labels $a_1< a_2<\dotsc<a_{\lfloor\frac{m+1}{2}\rfloor}$. We get the region $\overline{RO}_{m,n}(a_1,a_2,\dotsc, a_{\lfloor\frac{m+1}{2}\rfloor})$ (illustrated in Figure \ref{QAztec3}(b)).

We also start with the Aztec rectangle $\mathcal{AR}_{m,n}$. Again, we remove all squares along the \textit{southeastern}  side of the Aztec rectangle, and remove next the bottommost square the resulting region. Denote the just-obtained region by $R'$. We now remove all the even squares on the southwestern side of $R'$, and the $a_1, a_2, \dotsc,  a_{\lfloor\frac{m}{2}\rfloor}$ squares on the southeastern side.  We get the region $\overline{TE}_{m,n}(a_1,a_2,\dotsc, a_{\lfloor\frac{m}{2}\rfloor})$ (shown in Figure \ref{QAztec3}(c)). Do similarly, but remove the odd squares instead of the even squares on the southwestern side, we get the quartered Aztec rectangle $\overline{TO}_{m,n}(a_1,a_2,\dotsc, a_{\lfloor\frac{m-1}{2}\rfloor})$ (see Figure \ref{QAztec3}(d)).

We define two new function similar to $\E(...)$ and $\Od(...)$ as follows:
\begin{equation}\label{Eb}
\overline{\E}(a_1,a_2,\dotsc,a_n)=\frac{2^{n^2}a_1a_2\dotsc a_k}{0!2!4!\dotsc(2n-2)!}\prod_{1\leq i<j\leq n}(a_j-a_i)\prod_{1\leq i\leq j\leq n}(a_i+a_j),
\end{equation}
\begin{equation}\label{Ob}
\overline{\Od}(a_1,a_2,\dotsc,a_n)=\frac{2^{n^2}a_1a_2\dotsc a_k}{1!3!5!\dotsc(2n-1)!}\prod_{1\leq i<j\leq n}(a_j-a_i)\prod_{1\leq i< j\leq n}(a_i+a_j).
\end{equation}
 We have the following variant of Theorem \ref{main}.

\begin{theorem}\label{mainv} For any $1\leq k<n$ and $1\leq a_1<a_2<\dotsc<a_k\leq n$
\begin{equation}\label{VReq2}
\M(\overline{RE}_{2k,n}(a_1,a_2,\dotsc,a_k))=\M(\overline{RE}_{2k+1,n}(a_1,a_2,\dotsc,a_k))
=2^{k}\overline{\Od}(a_1,a_2,\dotsc,a_k)
\end{equation}
\begin{equation}\label{VReq1}
\M(\overline{RO}_{2k-1,n}(a_1,a_2,\dotsc,a_k))=\M(\overline{RO}_{2k,n}(a_1,a_2,\dotsc,a_k))
=2^{-k}\overline{\E}(a_1,a_2,\dotsc,a_k)
\end{equation}
\begin{equation}\label{VReq4}
\M(\overline{TE}_{2k,n}(a_1,a_2,\dotsc,a_{k}))=\M(\overline{TE}_{2k+1,n}(a_1,a_2,\dotsc,a_k))
=\overline{\Od}(a_1,a_2,\dotsc,a_k)
\end{equation}
\begin{equation}\label{VReq3}
\M(\overline{TO}_{2k+1,n}(a_1,a_2,\dotsc,a_{k}))=\M(\overline{TO}_{2k+2,n}(a_1,a_2,\dotsc,a_k))
=\frac{1}{(2k)!}\overline{\E}(a_1,a_2,\dotsc,a_k)
\end{equation}
\end{theorem}

The paper is organized as follows. In Section 2, we use subgraph replacement method to ``transform" the dual graphs of quartered Aztec rectangles to the dual graphs of  new families of regions, which we call \textit{quartered hexagons}. In Section 3, we use the classical methodology of Lindstr\"{o}m-Gessel-Viennot to enumerate the tilings of quartered hexagons. Finally, Section 4 gives the proofs of Theorem \ref{main} and \ref{mainv}.

\section{Subgraph replacement rules and quartered hexagons}

A \textit{perfect matching} of a graph $G$ is a collection of disjoint edges so that each vertex of $G$ is incident to exactly one edge of the collection. The \textit{dual graph} of a region $R$ is the graph whose vertices are the fundamental regions in $R$ and whose edges connect precisely two fundamental regions sharing an edge. The tilings of a regions are in bijection with the perfect matchings of its dual graph. By this point of view, we still use the notation $\M(G)$ for the number of perfect matchings of a graph.

Similar to the case of regions with weighted tiles, we can generalize the definition of the operation $\M(G)$ to the case of weighted graph $G$ as follows. The \textit{weight of a perfect matching} is defined to be the product of the weights of all constituent edges.  The operation $\M(G)$ is now defined to be the sum of the weights of all perfect matchings in $G$, and is called the \textit{matching generating function} of $G$. If $G$ does not have any perfect matching, we let $\M(G):=0$. In the weighted case, each edge of the dual graph carries the weight of the corresponding tile of the region, so the bijection mentioned in the previous paragraph is now weight-preserved.

An edge in a graph $G$ is called a \textit{forced edge}, if it is in every perfect matching of $G$. Let $G$ be a weighted graph with weight function $\wt$ on its edges, and $G'$ is obtained from $G$ by removing forced edges $e_1,\dotsc,e_k$, and removing the vertices incident to those edges. Then one clearly has
\begin{equation*}
\M(G)=\M(G')\prod_{i=1}^k\wt(e_i).
\end{equation*}
From now on, whenever we remove some forced edges, we remove also the vertices incident to them. We have the following fact by considering forced edges.

\begin{lemma}\label{lem2} For any $1\leq k<n$ and $1\leq a_1<a_2<\dotsc<a_k\leq n$
\begin{equation}\label{eq5}
\M(RE_{2k-1,n}(a_1,\dotsc,a_k))=\M(RE_{2k,n}(a_1,\dotsc,a_k)),
\end{equation}
\begin{equation}\label{eq6}
\M(RO_{2k,n}(a_1,\dotsc,a_k))=\M(RO_{2k+1,n}(a_1,\dotsc,a_k)),
\end{equation}
\begin{equation}\label{eq7}
\M(TE_{2k,n}(a_1,\dotsc,a_{k}))=\M(TE_{2k+1,n}(a_1,\dotsc,a_{k})),
\end{equation}
\begin{equation}\label{eq8}
\M(TO_{2k-1,n}(a_1,\dotsc,a_{k}))=\M(TO_{2k,n}(a_1,\dotsc,a_{k})).
\end{equation}
\end{lemma}

\begin{figure}\centering
\includegraphics[width=10cm]{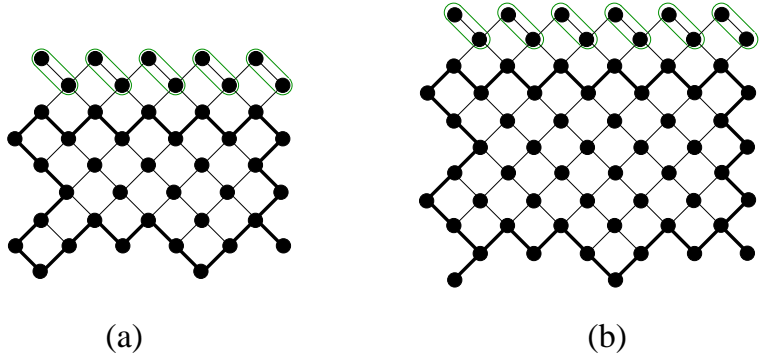}
\caption{ (a) Obtaining the dual graph of $RE_{3,5}(1,4)$ from the dual graph of $RE_{4,5}(1,4)$. (b) Obtaining the dual graph of $RO_{4,5}(1,4)$ from the dual graph of $RO_{5,5}(1,4)$. }
\label{force2}
\end{figure}

\begin{proof}
The proofs of the first two equalities are illustrated by Figures \ref{force2}(a) and (b), respectively; the forced edges are the circled ones on the top of the graphs. The last two equalities can be obtained similarly.
\end{proof}

Next, we will employ several basic preliminary results stated below.

\begin{figure}\centering
\begin{picture}(0,0)%
\includegraphics{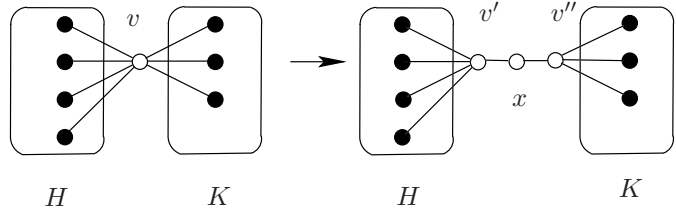}%
\end{picture}%
\setlength{\unitlength}{3947sp}%
\begingroup\makeatletter\ifx\SetFigFont\undefined%
\gdef\SetFigFont#1#2#3#4#5{%
  \reset@font\fontsize{#1}{#2pt}%
  \fontfamily{#3}\fontseries{#4}\fontshape{#5}%
  \selectfont}%
\fi\endgroup%
\begin{picture}(4188,1361)(593,-556)
\put(1336,591){\makebox(0,0)[lb]{\smash{{\SetFigFont{10}{12.0}{\familydefault}{\mddefault}{\updefault}{$v$}%
}}}}
\put(3549,621){\makebox(0,0)[lb]{\smash{{\SetFigFont{10}{12.0}{\familydefault}{\mddefault}{\updefault}{$v'$}%
}}}}
\put(3757, 89){\makebox(0,0)[lb]{\smash{{\SetFigFont{10}{12.0}{\familydefault}{\mddefault}{\updefault}{$x$}%
}}}}
\put(3999,621){\makebox(0,0)[lb]{\smash{{\SetFigFont{10}{12.0}{\familydefault}{\mddefault}{\updefault}{$v''$}%
}}}}
\put(820,-541){\makebox(0,0)[lb]{\smash{{\SetFigFont{10}{12.0}{\familydefault}{\mddefault}{\updefault}{$H$}%
}}}}
\put(1840,-535){\makebox(0,0)[lb]{\smash{{\SetFigFont{10}{12.0}{\familydefault}{\mddefault}{\updefault}{$K$}%
}}}}
\put(3031,-535){\makebox(0,0)[lb]{\smash{{\SetFigFont{10}{12.0}{\familydefault}{\mddefault}{\updefault}{$H$}%
}}}}
\put(4426,-484){\makebox(0,0)[lb]{\smash{{\SetFigFont{10}{12.0}{\familydefault}{\mddefault}{\updefault}{$K$}%
}}}}
\end{picture}%
\caption{Vertex splitting.}
\label{vertexsplitting}
\end{figure}

\begin{lemma} [Vertex-Splitting Lemma]\label{VS}
  Let $G$ be a graph, $v$ be a vertex of it, and denote the set of neighbors of $v$ by $N(v)$.
  For any disjoint union $N(v)=H\cup K$, let $G'$ be the graph obtained from $G\setminus v$ by including three new vertices $v'$, $v''$ and $x$ so that $N(v')=H\cup \{x\}$, $N(v'')=K\cup\{x\}$, and $N(x)=\{v',v''\}$ (see Figure \ref{vertexsplitting}). Then $\M(G)=\M(G')$.
\end{lemma}

\begin{lemma}[Star Lemma]\label{star}
Let $G$ be a weighted graph, and let $v$ be a vertex of~$G$. Let $G'$ be the graph obtained from $G$ by multiplying the weights of all edges that are incident to $v$ by $t>0$. Then $\M(G')=t\M(G)$.
\end{lemma}

Part (a) of the following result is a generalization due to Propp of the ``urban renewal" trick first observed by Kuperberg. Parts (b) and (c) are due to Ciucu (see Lemma 2.6 in [5]).

\begin{figure}\centering
\begin{picture}(0,0)%
\includegraphics{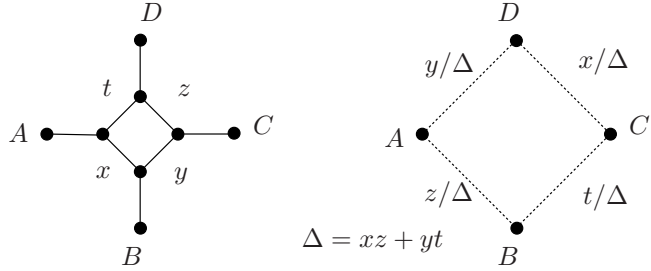}%
\end{picture}%
\setlength{\unitlength}{3947sp}%
\begingroup\makeatletter\ifx\SetFigFont\undefined%
\gdef\SetFigFont#1#2#3#4#5{%
  \reset@font\fontsize{#1}{#2pt}%
  \fontfamily{#3}\fontseries{#4}\fontshape{#5}%
  \selectfont}%
\fi\endgroup%
\begin{picture}(4054,1735)(340,-948)
\put(355,-156){\makebox(0,0)[lb]{\smash{{\SetFigFont{10}{12.0}{\familydefault}{\mddefault}{\updefault}{$A$}%
}}}}
\put(1064,-933){\makebox(0,0)[lb]{\smash{{\SetFigFont{10}{12.0}{\familydefault}{\mddefault}{\updefault}{$B$}%
}}}}
\put(1891,-106){\makebox(0,0)[lb]{\smash{{\SetFigFont{10}{12.0}{\familydefault}{\mddefault}{\updefault}{$C$}%
}}}}
\put(1182,603){\makebox(0,0)[lb]{\smash{{\SetFigFont{10}{12.0}{\familydefault}{\mddefault}{\updefault}{$D$}%
}}}}
\put(2717,-189){\makebox(0,0)[lb]{\smash{{\SetFigFont{10}{12.0}{\familydefault}{\mddefault}{\updefault}{$A$}%
}}}}
\put(3426,-933){\makebox(0,0)[lb]{\smash{{\SetFigFont{10}{12.0}{\familydefault}{\mddefault}{\updefault}{$B$}%
}}}}
\put(4253,-106){\makebox(0,0)[lb]{\smash{{\SetFigFont{10}{12.0}{\familydefault}{\mddefault}{\updefault}{$C$}%
}}}}
\put(3426,603){\makebox(0,0)[lb]{\smash{{\SetFigFont{10}{12.0}{\familydefault}{\mddefault}{\updefault}{$D$}%
}}}}
\put(904,-382){\makebox(0,0)[lb]{\smash{{\SetFigFont{10}{12.0}{\familydefault}{\mddefault}{\updefault}{$x$}%
}}}}
\put(1396,-388){\makebox(0,0)[lb]{\smash{{\SetFigFont{10}{12.0}{\familydefault}{\mddefault}{\updefault}{$y$}%
}}}}
\put(1418,130){\makebox(0,0)[lb]{\smash{{\SetFigFont{10}{12.0}{\familydefault}{\mddefault}{\updefault}{$z$}%
}}}}
\put(946,130){\makebox(0,0)[lb]{\smash{{\SetFigFont{10}{12.0}{\familydefault}{\mddefault}{\updefault}{$t$}%
}}}}
\put(2968,284){\makebox(0,0)[lb]{\smash{{\SetFigFont{10}{12.0}{\familydefault}{\mddefault}{\updefault}{$y/\Delta$}%
}}}}
\put(3934,311){\makebox(0,0)[lb]{\smash{{\SetFigFont{10}{12.0}{\familydefault}{\mddefault}{\updefault}{$x/\Delta$}%
}}}}
\put(3964,-544){\makebox(0,0)[lb]{\smash{{\SetFigFont{10}{12.0}{\familydefault}{\mddefault}{\updefault}{$t/\Delta$}%
}}}}
\put(2965,-526){\makebox(0,0)[lb]{\smash{{\SetFigFont{10}{12.0}{\familydefault}{\mddefault}{\updefault}{$z/\Delta$}%
}}}}
\put(2197,-817){\makebox(0,0)[lb]{\smash{{\SetFigFont{10}{12.0}{\familydefault}{\mddefault}{\updefault}{$\Delta= xz+yt$}%
}}}}
\end{picture}%
\caption{Urban renewal.}
\label{spider1}
\end{figure}

\begin{figure}\centering
\resizebox{!}{3.7cm}{
\begin{picture}(0,0)%
\includegraphics{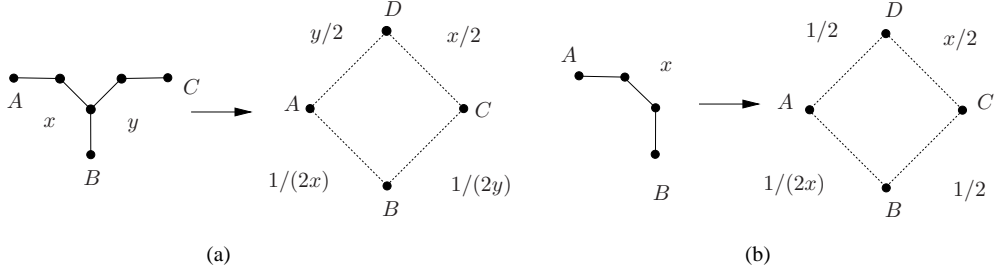}%
\end{picture}%
\setlength{\unitlength}{3947sp}%
\begingroup\makeatletter\ifx\SetFigFont\undefined%
\gdef\SetFigFont#1#2#3#4#5{%
  \reset@font\fontsize{#1}{#2pt}%
  \fontfamily{#3}\fontseries{#4}\fontshape{#5}%
  \selectfont}%
\fi\endgroup%
\begin{picture}(7605,2136)(222,-1352)
\put(237,-106){\makebox(0,0)[lb]{\smash{{\SetFigFont{10}{12.0}{\familydefault}{\mddefault}{\updefault}{$A$}%
}}}}
\put(828,-697){\makebox(0,0)[lb]{\smash{{\SetFigFont{10}{12.0}{\familydefault}{\mddefault}{\updefault}{$B$}%
}}}}
\put(1592,  0){\makebox(0,0)[lb]{\smash{{\SetFigFont{10}{12.0}{\familydefault}{\mddefault}{\updefault}{$C$}%
}}}}
\put(519,-249){\makebox(0,0)[lb]{\smash{{\SetFigFont{10}{12.0}{\familydefault}{\mddefault}{\updefault}{$x$}%
}}}}
\put(1164,-264){\makebox(0,0)[lb]{\smash{{\SetFigFont{10}{12.0}{\familydefault}{\mddefault}{\updefault}{$y$}%
}}}}
\put(2566,426){\makebox(0,0)[lb]{\smash{{\SetFigFont{10}{12.0}{\familydefault}{\mddefault}{\updefault}{$y/2$}%
}}}}
\put(3616,396){\makebox(0,0)[lb]{\smash{{\SetFigFont{10}{12.0}{\familydefault}{\mddefault}{\updefault}{$x/2$}%
}}}}
\put(2244,-714){\makebox(0,0)[lb]{\smash{{\SetFigFont{10}{12.0}{\rmdefault}{\mddefault}{\updefault}{$1/(2x)$}%
}}}}
\put(3646,-744){\makebox(0,0)[lb]{\smash{{\SetFigFont{10}{12.0}{\rmdefault}{\mddefault}{\updefault}{$1/(2y)$}%
}}}}
\put(2363,-129){\makebox(0,0)[lb]{\smash{{\SetFigFont{10}{12.0}{\familydefault}{\mddefault}{\updefault}{$A$}%
}}}}
\put(3131,600){\makebox(0,0)[lb]{\smash{{\SetFigFont{10}{12.0}{\familydefault}{\mddefault}{\updefault}{$D$}%
}}}}
\put(3830,-174){\makebox(0,0)[lb]{\smash{{\SetFigFont{10}{12.0}{\rmdefault}{\mddefault}{\updefault}{$C$}%
}}}}
\put(3119,-933){\makebox(0,0)[lb]{\smash{{\SetFigFont{10}{12.0}{\rmdefault}{\mddefault}{\updefault}{$B$}%
}}}}
\put(4489,249){\makebox(0,0)[lb]{\smash{{\SetFigFont{10}{12.0}{\rmdefault}{\mddefault}{\updefault}{$A$}%
}}}}
\put(5198,-814){\makebox(0,0)[lb]{\smash{{\SetFigFont{10}{12.0}{\rmdefault}{\mddefault}{\updefault}{$B$}%
}}}}
\put(6150,-115){\makebox(0,0)[lb]{\smash{{\SetFigFont{10}{12.0}{\rmdefault}{\mddefault}{\updefault}{$A$}%
}}}}
\put(6977,-942){\makebox(0,0)[lb]{\smash{{\SetFigFont{10}{12.0}{\rmdefault}{\mddefault}{\updefault}{$B$}%
}}}}
\put(7686,-115){\makebox(0,0)[lb]{\smash{{\SetFigFont{10}{12.0}{\rmdefault}{\mddefault}{\updefault}{$C$}%
}}}}
\put(6977,594){\makebox(0,0)[lb]{\smash{{\SetFigFont{10}{12.0}{\rmdefault}{\mddefault}{\updefault}{$D$}%
}}}}
\put(5251,179){\makebox(0,0)[lb]{\smash{{\SetFigFont{10}{12.0}{\rmdefault}{\mddefault}{\updefault}{$x$}%
}}}}
\put(6369,441){\makebox(0,0)[lb]{\smash{{\SetFigFont{10}{12.0}{\rmdefault}{\mddefault}{\updefault}{$1/2$}%
}}}}
\put(7501,-781){\makebox(0,0)[lb]{\smash{{\SetFigFont{10}{12.0}{\rmdefault}{\mddefault}{\updefault}{$1/2$}%
}}}}
\put(6046,-736){\makebox(0,0)[lb]{\smash{{\SetFigFont{10}{12.0}{\rmdefault}{\mddefault}{\updefault}{$1/(2x)$}%
}}}}
\put(7426,374){\makebox(0,0)[lb]{\smash{{\SetFigFont{10}{12.0}{\rmdefault}{\mddefault}{\updefault}{$x/2$}%
}}}}
\end{picture}}
\caption{Two variants of the urban renewal trick.}
\label{spider2}
\end{figure}

\begin{lemma} [Spider Lemma]\label{spider}
(a) Let $G$ be a weighted graph containing the subgraph $K$ shown on the left in Figure \ref{spider1} (the labels indicate weights, unlabeled edges have weight 1). Suppose in addition that the four inner black vertices in the subgraph $K$, different from $A,B,C,D$, have no neighbors outside $K$. Let $G'$ be the graph obtained from $G$ by replacing $K$ by the graph $\overline{K}$ shown on right in Figure \ref{spider1}, where the dashed lines indicate new edges, weighted as shown. Then $\M(G)=(xz+yt)\M(G')$.

(b) Consider the above local replacement operation when $K$ and $\overline{K}$ are graphs shown in Figure \ref{spider2}(a) with the indicated weights (in particular, $K'$ has a new vertex $D$, that is incident only to $A$ and $C$). Then $\M(G)=2\M(G')$.

(c) The statement of part (b) is also true when $K$ and $\overline{K}$ are the graphs indicated in Figure \ref{spider2}(b) (in this case $G'$ has two new vertices $C$ and $D$, they are adjacent only to one another and to $B$ and $A$, respectively).
\end{lemma}

\begin{figure}\centering
\begin{picture}(0,0)%
\includegraphics{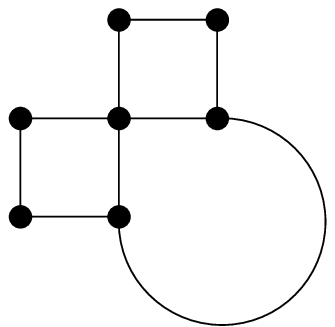}%
\end{picture}%
\setlength{\unitlength}{3947sp}%
\begingroup\makeatletter\ifx\SetFigFont\undefined%
\gdef\SetFigFont#1#2#3#4#5{%
  \reset@font\fontsize{#1}{#2pt}%
  \fontfamily{#3}\fontseries{#4}\fontshape{#5}%
  \selectfont}%
\fi\endgroup%
\begin{picture}(1725,1926)(931,-1107)
\put(1418, 12){\makebox(0,0)[lb]{\smash{{\SetFigFont{12}{14}{\rmdefault}{\mddefault}{\updefault}{$a$}%
}}}}
\put(1655,603){\makebox(0,0)[lb]{\smash{{\SetFigFont{12}{14}{\rmdefault}{\mddefault}{\updefault}{$b_1$}%
}}}}
\put(2245,603){\makebox(0,0)[lb]{\smash{{\SetFigFont{12}{14}{\rmdefault}{\mddefault}{\updefault}{$b_2$}%
}}}}
\put(2363, 12){\makebox(0,0)[lb]{\smash{{\SetFigFont{12}{14}{\rmdefault}{\mddefault}{\updefault}{$b_3$}%
}}}}
\put(969, 36){\makebox(0,0)[lb]{\smash{{\SetFigFont{12}{14}{\rmdefault}{\mddefault}{\updefault}{$c_1$}%
}}}}
\put(946,-696){\makebox(0,0)[lb]{\smash{{\SetFigFont{12}{14}{\rmdefault}{\mddefault}{\updefault}{$c_2$}%
}}}}
\put(1479,-819){\makebox(0,0)[lb]{\smash{{\SetFigFont{12}{14}{\rmdefault}{\mddefault}{\updefault}{$c_3$}%
}}}}
\end{picture}
\caption{Illustrating Lemma \ref{4cycle}.}
\label{4cyclelemma}
\end{figure}

\begin{lemma}[\cite{Ciucu1}, Lemma 4.2]\label{4cycle}
Let $G$ be a weighted graph having a $7$-vertex subgraph $H$ consisting of two $4$-cycles that share a vertex. Let $a$, $b_1$, $b_2$, $b_3$ and $a$, $c_1$, $c_2$, $c_3$ be the vertices of the 4-cycles (listed in cyclic order) and suppose $b_3$ and $c_3$ are only the vertices of $H$ with the neighbors outside $H$. Let $G'$ be the subgraph of $G$ obtained by deleting $b_1$, $b_2$, $c_1$ and $c_2$, weighted by restriction. Then if the product of weights of opposite edges in each $4$-cycle of $H$ is constant, we have
\[\M(G)=2wt(b_1,b_2)wt(c_1,c_2) \M(G').\]
\end{lemma}

By the above fundamental lemmas, we have the following fact.

\begin{lemma}\label{lem1} For any $1\leq k<n$ and $1\leq a_1<a_2<\dotsc<a_k\leq n$
\begin{equation}\label{eq1}
\M(RE_{2k-1,n}(a_1,\dotsc,a_k))=2^{k}\M(TO_{2k-1,n}(a_1,\dotsc,a_{k}))
\end{equation}
 and
\begin{equation}\label{eq2}
\M(RO_{2k,n}(a_1,\dotsc,a_k))=2^{k}\M(TE_{2k,n}(a_1,\dotsc,a_{k})).
\end{equation}
\end{lemma}

\begin{figure}\centering
\includegraphics[width=13.5cm]{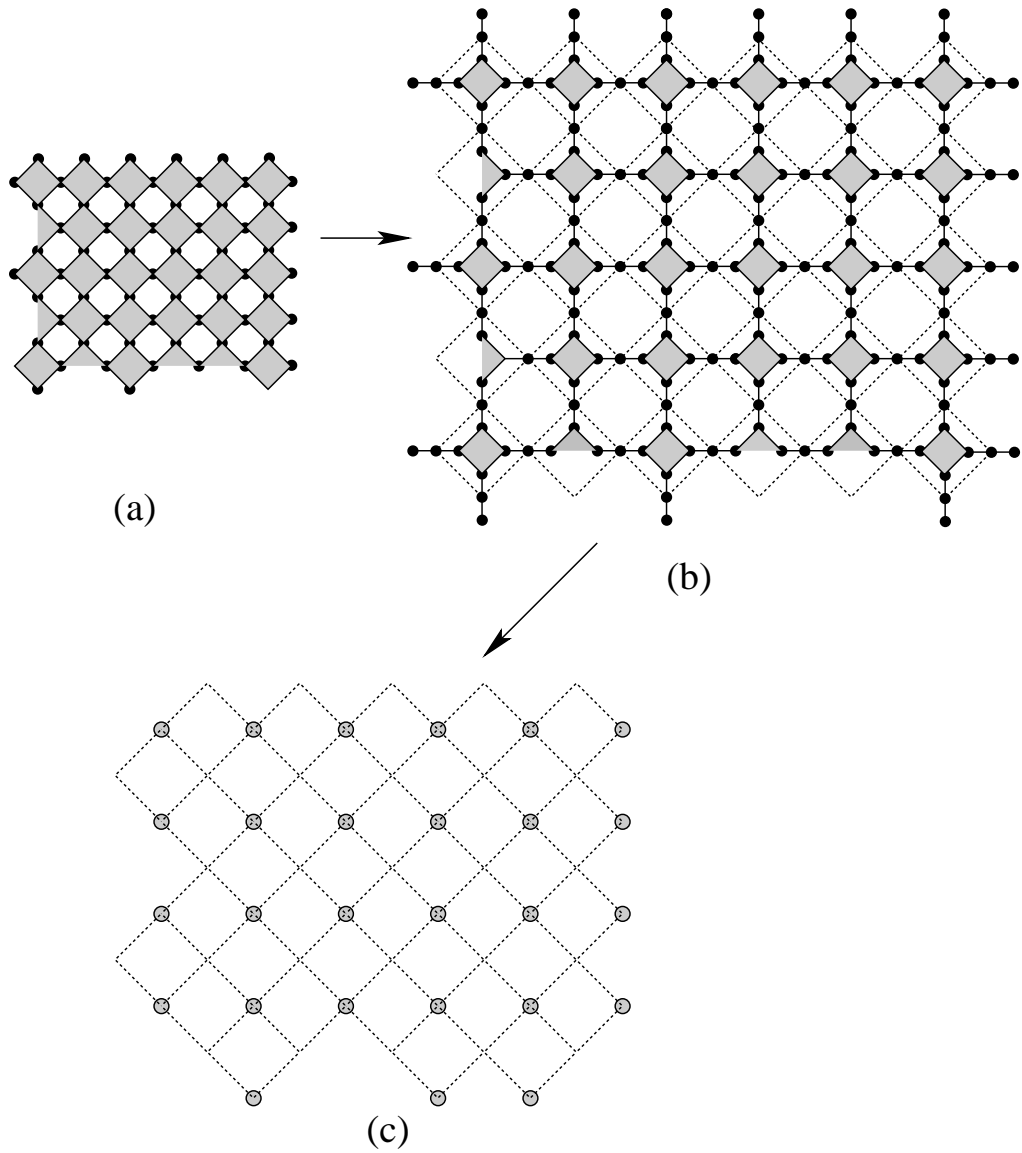}
\caption{Illustrating the proof of Lemma \ref{lem1}}
\label{TransformQR2}
\end{figure}

\begin{proof}
The proofs of (\ref{eq1}) and (\ref{eq2}) are essentially the same, so we present only the proof of (\ref{eq1}).

The proof of (\ref{eq1}) is illustrated in Figure \ref{TransformQR2}, for the case $k=3$, $n=6$, $a_1=1$, $a_2=3$, and $a_3=6$. First, we apply Vertex-splitting Lemma \ref{VS} to all vertices of the dual graph $G$ of $RE_{2k-1,n}(a_1,\dotsc,a_k)$ (see Figures \ref{TransformQR2}(a) and (b)). Second, apply the suitable replacements in Spider Lemma \ref{spider} at  $(2k-1)n$ diamond cells and partial cells with legs (they are replaced by dotted diamond with edge-weight $1/2$). Next, we removed all edges adjacent to a vertex of degree one (which are forced). We get a weighted version $G'$ of the dual graph of $TO_{2k-1,n}(a_1,\dotsc,a_{k})$, where all edges have weight $1/2$ (see Figures \ref{TransformQR2}(b) and (c)). Finally, we apply Star Lemma \ref{star} (with factor $t=2$) at all $(2k-1)n-k$ shaded vertices in the resulting graph, and get the dual graph $G''$ of $TO_{2k-1,n}(a_1,\dotsc,a_{k})$. By Lemmas \ref{VS}, \ref{star} and \ref{spider}, we obtain
\begin{equation}
\M(G)=2^{(2k-1)n}\M(G')=2^{(2k-1)n}2^{-(2k-1)n+k}\M(G'),
\end{equation}
which implies (\ref{eq1}).
\end{proof}

\begin{remark}
One can prove the equality (\ref{eq1}) by apply Ciucu's Complementation Theorem in \cite{Ciucu2}. We notice that \textit{cellular completion} (defined in \cite{Ciucu2}) of the dual graph of the region $RE_{2k-1,n}(a_1,\dotsc,a_k)$ is the graph $G'$ in the proof of Lemma \ref{lem1}. Moreover, each perfect matching of $G'$ consists of exactly $(2k-1)n-k$ edges of weight $1/2$, so \[\M(G')=2^{-(2k-1)n+k}\M(TO_{2k-1,n}(a_1,\dotsc,a_{k})).\]
\end{remark}

The \textit{connected sum} $G\#G'$ of two disjoint graphs $G$ and $G'$ along the ordered sets of vertices $\{v_1,\dotsc,v_n\}\subset V(G)$ and $\{v'_1,\dotsc,v'_n\}\subset V(G')$ is the graph obtained from $G$ and $G'$ by identifying vertices $v_i$ and $v'_i$, for $i=1,\dotsc,n$.
\begin{figure}\centering
\includegraphics[width=8cm]{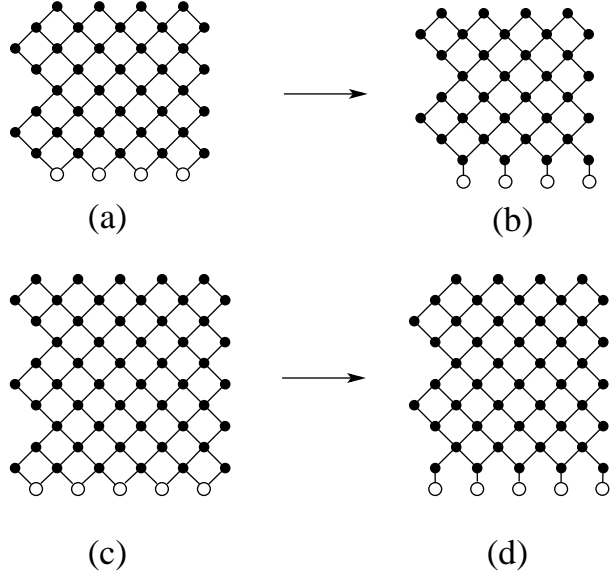}
\caption{Two transformations in Lemma \ref{transform}.}
\label{TransformQR}
\end{figure}

\begin{lemma}\label{transform}
Let $G$ be a graph, and let $\{v_1,v_2,\dotsc,v_n\}$ be an ordered subset of its vertex set.

(a) Assume that $K$ is the graph obtained from the dual graph of $\mathcal{TR}_{2q+1,n+1}$ by removing all bottommost vertices and the odd vertices on the leftmost column; $K'$ is obtained from the dual graph of $\mathcal{AR}_{2q,n}$ by removing all bottommost vertices and the odd vertices on the leftmost column (ordered from bottom to top), and appending $n$ vertical edges to the bottom of the resulting graph. Then
\begin{equation}\label{trans1}
\M(G\#K)=2^{q}\M(G\#K').
\end{equation}
The transformation is illustrated in Figures \ref{TransformQR}(a) and (b), for $q=2$ and $n=4$; the white circles indicate the vertices $\{v_1,v_2,\dotsc,v_n\}$.

(b) Assume $H$ is the graph obtained from the dual graph of  $\mathcal{AR}_{2q+1,n}$ by removing all even vertices on the leftmost column (ordered from bottom to top); $H'$ is the graph obtained from the dual graph of $\mathcal{TR}_{2q+1,n}$ by removing odd vertices on the leftmost column, and appending $n$ vertical edges to the bottom of the resulting graph. Then
\begin{equation}\label{trans2}
\M(G\#H)=2^{q}\M(G\#H').
\end{equation}
The transformation is illustrated in Figures \ref{TransformQR}(c) and (d), for $q=2$ and $n=5$; the white circles indicate the vertices $\{v_1,v_2,\dotsc,v_n\}$.

In the two equalities (\ref{trans1}) and (\ref{trans2}), the connected sum acts on $G$ along $\{v_1,v_2,\dotsc,v_n\}$ and acts on the other two summands along their bottom vertices (ordered from left to right).
\end{lemma}

\begin{figure}\centering
\includegraphics[width=13.5cm]{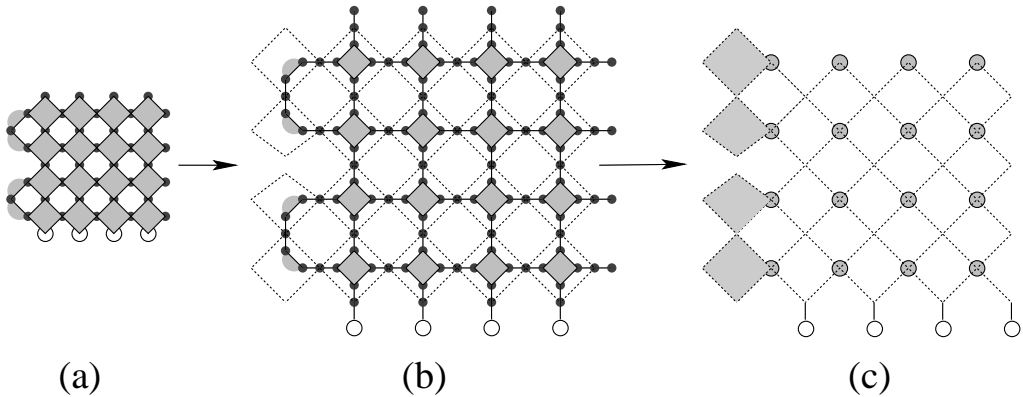}
\caption{Illustrating the proof of Lemma \ref{transform}(a).}
\label{TransformQR1}
\end{figure}

\begin{proof}
Since the proofs of parts (a) and (b) are essentially the same, the proof of part (b) is omitted.

The illustration of the proof of part (a) is shown in Figure \ref{TransformQR1}, for $q=2$ and $n=4$. First, we apply Vertex-splitting Lemma \ref{VS} to all vertices of $K$. We get the graph with solid edges in Figure \ref{TransformQR1}(b). Second, apply the suitable replacements in Spider Lemma \ref{spider}  to  $2q(n+1)$ diamond cells and partial cells with legs in the resulting graph, and remove all edges adjacent to a vertex of degree 1 (which are forced). We get the graph in the Figure \ref{TransformQR1}(c), the dotted edges are weighted by 1/2.  Third, apply Lemma \ref{4cycle} to all $q$ 7-vertex subgraphs consisting of two shaded 4-cycles, and apply Star Lemma \ref{star} with factor $t=2$ at all $2qn$ shaded vertices as in Figure \ref{TransformQR1}(c). We get finally the graph $G\#K'$. By Lemmas \ref{VS}, \ref{star}, \ref{spider} and \ref{4cycle}, we obtain
\[\M(G\#K)=2^{2q(n+1)}2^{-2qn}2^{q}\M(G\#K'),\]
which implies (\ref{trans1}).
\end{proof}

\bigskip

Denote by $H_{a,b,c}$ the (lozenge) hexagon of sides $a,b,c,a,b,c$ (in cyclic order, starting from the northwestern side) in the triangular lattice. In the spirit of quartered Aztec rectangles, we introduce four new families of regions, which we call \textit{quartered hexagons}, that will play a key role in our proofs of Theorems \ref{main} and \ref{mainv}.

Divide the hexagon $H_{m,2(n-k)+1,m}$, where $k=\lfloor\frac{m+1}{2}\rfloor$, into four equal parts by its vertical and horizontal symmetry axes (see Figure \ref{QHnew} for an example). We consider the portion of of the hexagon that consists of unit triangles lying completely inside the upper right quarter. Remove the $a_1$-st, the $a_2$-nd, $\dotsc,$ and the  $a_k$-th up-pointing unit triangles (ordered from left to right) from the bottom of the portion. Denote by $QH_{m,n}(a_1,a_2,\dotsc,a_k)$ the resulting region. See the region restricted in the bold contour in Figure \ref{QHnew} for an example with $k=7$, $m=13$, $n=12$, $a_1=2$, $a_2=3$, $a_3=5$, $a_4=7$, $a_5=8$, $a_6=10$, $a_7=12$; and  Figure \ref{QH}(a) shows an example, for $k=6$, $m=12$, $n=11$, $a_1=2$, $a_2=3$, $a_3=5$, $a_4=6$, $a_5=8$, $a_6=11$.

Next, we consider the variant  quartered hexagon obtained from $QH_{2k,n}(a_1,a_2,\dotsc,a_k)$ obtained by assigning all $k$ vertical rhombus on its left side a weight 1/2 (see Figure \ref{QH}(c)). Denote the resulting region by  $\overline{QH}_{2k,n}(a_1,a_2,\dotsc,a_k)$.
 We consider another variant of quartered hexagons as follows. We assign all $k-1$ vertical rhombus on the left side of  upper right quarter of the hexagon $H_{m,2(n-k)+1,m}$  a weight $1/2$, and remove the leftmost up-pointing unit triangle from the bottom of the region. Next, we remove the $a_1$-st, the $a_2$-nd, $\dotsc,$ and the  $a_{k-1}$-th up-pointing unit triangles from the bottom of the resulting region. The new region is denoted by $\overline{QH}_{2k-1,n}(a_1,a_2,\dotsc,a_{k-1})$ (illustrated in Figure \ref{QH}(b)).

\begin{figure}\centering
\includegraphics[width=11cm]{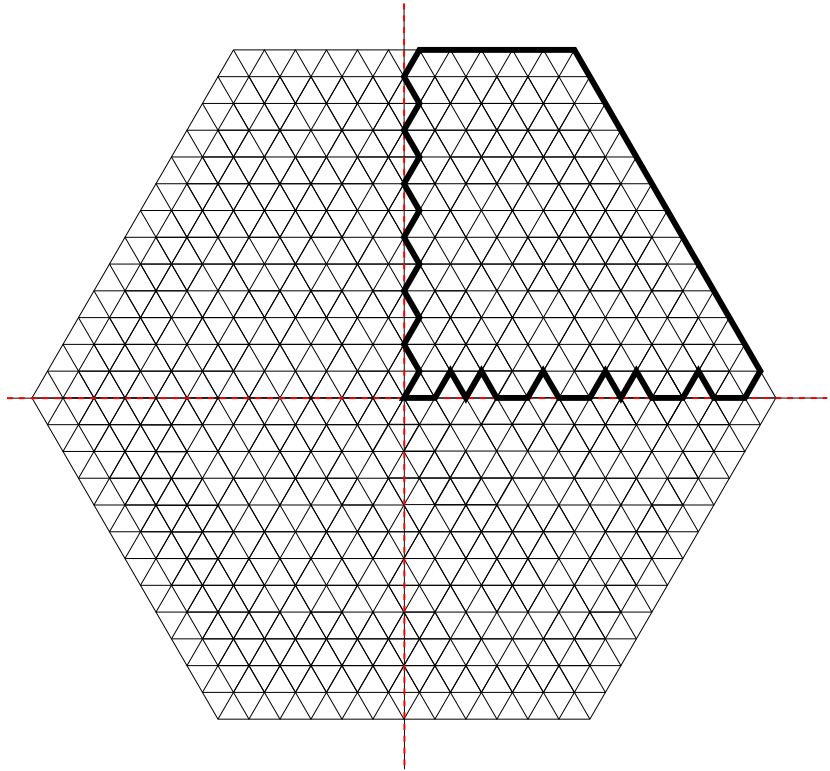}
\caption{The hexagon $H_{13,11,13}$ and the region $QH_{13,12}(2,3,5,7,8,10,12)$ (restricted by the bold contour). }
\label{QHnew}
\end{figure}

\begin{figure}\centering
\includegraphics[width=12cm]{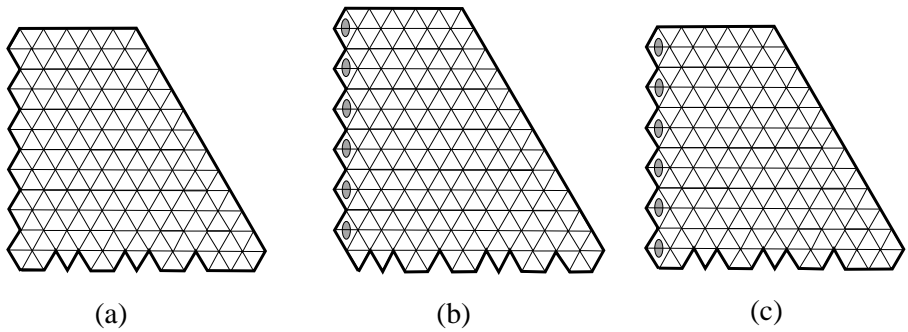}
\caption{Three quartered hexagons:  $(a)\quad QH_{12,11}(2,3,5,6,8,11)$, $(b)\quad \overline{QH}_{13,12}(1,4,6,7,9,11)$, and $(c)\quad \overline{QH}_{12,11}(2,3,5,6,8,11)$. }
\label{QH}
\end{figure}

The connection between the numbers of tilings of quartered Aztec rectangles and quartered hexagons is given by the following lemma.
\begin{lemma}\label{QHex2} For $1\leq k<n$ and $1\leq a_1<a_2<\dotsc<a_k\leq n$
\begin{equation}\label{eq3}
\M(TO_{2k-1,n}(a_1,a_2,\dotsc,a_{k}))=2^{k(k-1)}\M(QH_{2k-1,n}(a_1,a_2,\dotsc,a_k))
\end{equation}
and
\begin{equation}\label{QHex3}
\M( \overline{TE}_{2k,n}(a_1,a_2,\dotsc,a_k))=2^{k^2}\M(QH_{2k,n}(a_1,a_2,\dotsc,a_k)).
\end{equation}
\end{lemma}

\begin{figure}\centering
\includegraphics[width=13.6cm]{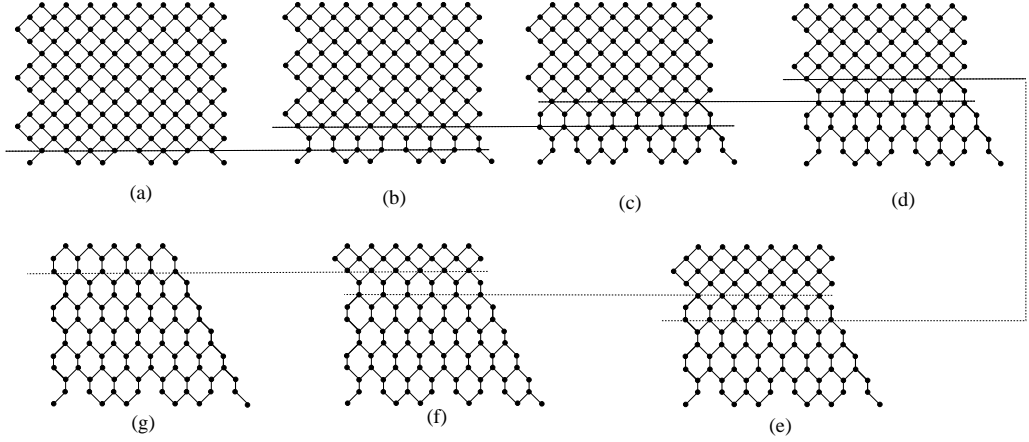}
\caption{Illustrating the proof of Theorem \ref{QHex2}.}
\label{TransformQH2}
\end{figure}

\begin{proof}
We prove the equality (\ref{eq3}) first.
We use the $(2k-2)$-step transforming process consisting alternatively the transformation  in Lemma \ref{transform} parts (a) and (b), for $q=k-1,\dotsc,2,1$, and starting by the transformation in part (a),
to transform the dual graph of $TO_{2k-1,n}(a_1,a_2,$ $\dotsc,a_{k})$ to the dual  graph  of $QH_{k,n}(a_1,a_2,\dotsc,a_k)$  (illustrated  in Figures \ref{TransformQH2}(a)--(g); the part above the top dotted line in a graph is replaced by the part above that line in the next graph).  By Lemma \ref{transform}, we get
\begin{equation}
\frac{\M(TO_{2k-1,n}(a_1,a_2,\dotsc,a_{k}))}{\M(QH_{2k-1,n}(a_1,a_2,\dotsc,a_k))}=2^{\sum_{i=1}^{k-1}(i+i)}=2^{k(k-1)},
\end{equation}
which implies (\ref{eq3}).

Similarly, we can get (\ref{QHex3}) by using a $(2k-1)$-step transforming process consisting alternatively the transformations in Lemma \ref{transform} parts (b) and (a), and starting by the transformation in part (b) to transform the dual graph of the region on the left hand side to dual graph of the region on the right-hand side.
\end{proof}

\section{Enumeration of tilings of quartered hexagons}
The main goal of this section is to enumerate the tilings of quartered hexagons using Lindstr\"{o}m-Gessel-Viennot  methodology.

The numbers of tilings of quartered hexagons are given by the following theorem.
\begin{theorem}\label{QHex} For any $1\leq k<n$ and $1\leq a_1<a_2<\dotsc<a_k\leq n$
\begin{equation}\label{QHeq1}
\M(QH_{2k-1,n}(a_1,a_2,\dotsc,a_k))=2^{-k^2}\E(a_1,a_2,\dotsc,a_k),
\end{equation}
\begin{equation}\label{QHeq2}
\M(QH_{2k,n}(a_1,a_2,\dotsc,a_k))=2^{-k^2}\overline{\Od}(a_1,a_2,\dotsc,a_k),
\end{equation}
\begin{equation}\label{QHeq3}
\M(\overline{QH}_{2k+1,n}(a_1,a_2,\dotsc,a_k))=\frac{2^{-k(k+1)}}{(2k)!}\overline{\E}(a_1,a_2,\dotsc,a_k),
\end{equation}
\begin{equation}\label{QHeq4}
\M(\overline{QH}_{2k,n}(a_1,a_2,\dotsc,a_k))=2^{-k(k+1)}\Od(a_1,a_2,\dotsc,a_k).
\end{equation}
\end{theorem}

We consider a useful factorization theorem due to Ciucu \cite{Ciucu1}, which we will employ in the proof of Theorem \ref{QHex}.

\begin{figure}\centering
\begin{picture}(0,0)%
\includegraphics{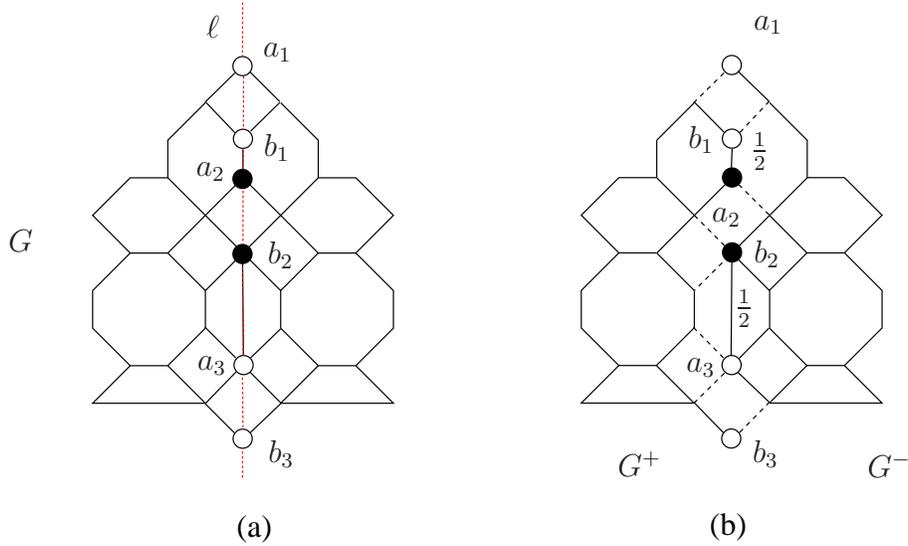}%
\end{picture}%
\setlength{\unitlength}{3947sp}%
\begingroup\makeatletter\ifx\SetFigFont\undefined%
\gdef\SetFigFont#1#2#3#4#5{%
  \reset@font\fontsize{#1}{#2pt}%
  \fontfamily{#3}\fontseries{#4}\fontshape{#5}%
  \selectfont}%
\fi\endgroup%
\begin{picture}(5586,3473)(1111,-2672)
\put(5788,-224){\makebox(0,0)[lb]{\smash{{\SetFigFont{12}{14}{\rmdefault}{\mddefault}{\updefault}{$\frac{1}{2}$}%
}}}}
\put(5691,-1211){\makebox(0,0)[lb]{\smash{{\SetFigFont{12}{14}{\rmdefault}{\mddefault}{\updefault}{$\frac{1}{2}$}%
}}}}
\put(2724,419){\makebox(0,0)[lb]{\smash{{\SetFigFont{12}{14}{\rmdefault}{\mddefault}{\updefault}{$a_1$}%
}}}}
\put(2731,-204){\makebox(0,0)[lb]{\smash{{\SetFigFont{12}{14}{\rmdefault}{\mddefault}{\updefault}{$b_1$}%
}}}}
\put(2754,-879){\makebox(0,0)[lb]{\smash{{\SetFigFont{12}{14}{\rmdefault}{\mddefault}{\updefault}{$b_2$}%
}}}}
\put(2304,-331){\makebox(0,0)[lb]{\smash{{\SetFigFont{12}{14}{\rmdefault}{\mddefault}{\updefault}{$a_2$}%
}}}}
\put(2312,-1550){\makebox(0,0)[lb]{\smash{{\SetFigFont{12}{14}{\rmdefault}{\mddefault}{\updefault}{$a_3$}%
}}}}
\put(2758,-2128){\makebox(0,0)[lb]{\smash{{\SetFigFont{12}{14}{\rmdefault}{\mddefault}{\updefault}{$b_3$}%
}}}}
\put(5806,591){\makebox(0,0)[lb]{\smash{{\SetFigFont{12}{14}{\rmdefault}{\mddefault}{\updefault}{$a_1$}%
}}}}
\put(5396,-155){\makebox(0,0)[lb]{\smash{{\SetFigFont{12}{14}{\rmdefault}{\mddefault}{\updefault}{$b_1$}%
}}}}
\put(5544,-616){\makebox(0,0)[lb]{\smash{{\SetFigFont{12}{14}{\rmdefault}{\mddefault}{\updefault}{$a_2$}%
}}}}
\put(5807,-876){\makebox(0,0)[lb]{\smash{{\SetFigFont{12}{14}{\rmdefault}{\mddefault}{\updefault}{$b_2$}%
}}}}
\put(5383,-1560){\makebox(0,0)[lb]{\smash{{\SetFigFont{12}{14}{\rmdefault}{\mddefault}{\updefault}{$a_3$}%
}}}}
\put(5799,-2130){\makebox(0,0)[lb]{\smash{{\SetFigFont{12}{14}{\rmdefault}{\mddefault}{\updefault}{$b_3$}%
}}}}
\put(1126,-804){\makebox(0,0)[lb]{\smash{{\SetFigFont{12}{14}{\rmdefault}{\mddefault}{\updefault}{$G$}%
}}}}
\put(4951,-2214){\makebox(0,0)[lb]{\smash{{\SetFigFont{12}{14}{\rmdefault}{\mddefault}{\updefault}{$G^+$}%
}}}}
\put(6519,-2221){\makebox(0,0)[lb]{\smash{{\SetFigFont{12}{14}{\rmdefault}{\mddefault}{\updefault}{$G^-$}%
}}}}
\put(2368,539){\makebox(0,0)[lb]{\smash{{\SetFigFont{12}{14}{\rmdefault}{\mddefault}{\updefault}{$\ell$}%
}}}}
\end{picture}%
\caption{(a) A graph $G$ with symmetric axis; (b) the resulting graph after the cutting procedure.}
\label{verticalfactor}
\end{figure}

Let $G$ be a weighted planar bipartite graph that is symmetric about a vertical line $\ell$. Assume that the set of vertices lying on $\ell$ is a cut set of $G$ (i.e., the removal of these vertices disconnects $G$). One readily sees that the number of vertices of $G$ on $\ell$ must be even if $G$ has perfect matchings, let $w(G)$ be half of this number. Let $a_1,b_1,a_2,b_2,\dots,a_{w(G)},b_{w(G)}$ be the vertices lying on $\ell$, as they occur from top to bottom. Let us color vertices of $G$ by black or white so that any two adjacent vertices have opposite colors. Without loss of generality, we assume that $a_1$ is always colored white. Delete all edges on the left of $\ell$ at all white $a_i$'s and black $b_j$'s, and delete all edges on the right of  $\ell$ at all black $a_i$'s and white $b_j$'s. Reduce the weight of each edge lying on $\ell$ by half; leave all other weights unchanged. Since the set of vertices of $G$ on $\ell$ is a cut set, the graph obtained from the above procedure has two disconnected parts, one on the left of $\ell$ and one on the right of $\ell$, denoted by $G^+$ and $G^-$ respectively (see Figure \ref{verticalfactor}).

 \begin{theorem}[Factorization Theorem, Ciucu \cite{Ciucu1}]
Let $G$ be a bipartite weighted symmetric graph separated by its symmetry axis. Then
\begin{equation}\label{factoreq}
\M(G)=2^{w(G)}\M(G^+)\M(G^-).
\end{equation}
\end{theorem}

Next, we quote a result on the number of tilings of a \textit{semi-hexagon} due to The next result is due to Cohn, Larsen and Propp (see \cite{Cohn}, Proposition 2.1). A semi-hexagon of sides $a,b,a,a+b$ is the portion of a hexagon of sides $a$, $b$, $a$, $a$, $b$, $a$ (in cyclic order, starting from the northwestern side) in the triangular lattice that stays above the horizontal symmetric axis of the hexagon. We are interested in the number of tilings of the semi-hexagon sides $a,b,a,a+b$, where the $s_1$-st, the $s_2$-nd, $\dotsc$, and the $s_a$-th up-pointing unit triangles in the base have been removed, denoted by $SH_{a,b}(s_1,s_2,\dotsc,s_a)$.

\begin{lemma}\label{semihex}
For any $a,b>0$, and $1\leq s_1<s_2<\dotsc<s_a\leq a+b$
\begin{equation}
\M(SH_{a,b}(s_1,s_2,\dotsc,s_a))=\prod_{1\leq i< j\leq a}\frac{s_j-s_i}{j-i}.
\end{equation}
\end{lemma}

The following determinant identity has been proved by Krattenthaler \cite{Krat2}.
\begin{lemma}[\cite{Krat2}, Identity (2.10) in Lemma 4] \label{Klem}
Let $X_1,X_2,\dots,X_n$, $A_2,\dots,A_n$ be indeterminates, and let $C$ be a constant.  Then
\begin{align}
&\det_{1\leq i,j\leq n}((X_i-A_n-C)(X_i-A_{n-1}-C)\dotsc(X_i-A_{j+1}-C) \notag\\
&\cdot(X_i+A_n)(X_i+A_{n-1})\dotsc(X_i+A_{j+1}))=\prod_{1\leq i<j\leq n}(X_j-X_i)(C-X_i-X_j).
\end{align}
\end{lemma}

We are now ready to prove Theorem \ref{QHex}.
\begin{proof}[Proof of Theorem \ref{QHex}]
Write for short $R_1:=QH_{2k-1,n}(a_1,a_2,\dotsc,a_k)$.
We use a standard bijection mapping each tiling $\mu$ of the region $R_1$ in the triangular lattice to a $k$-tuple of non-intersection lattice paths taking steps west or north on the square grid $\mathbb{Z}^2$.

Label the centers of the left sides of up-pointing unit triangles along the left boundary of $R_1$ from bottom to top by $v_1,v_2,\dotsc,v_k$. Label the centers of  the left sides of  up-pointing unit triangles, which have been removed from the bottom of the region, from left to right by $u_1,u_2,\dotsc,u_k$ (see Figure \ref{QH2}(a) for an example corresponding to the region in Figure \ref{QHnew}; the black dots indicate the points $u_i$'s and $v_j$'s).

Consider now a rhombus  $r_1$ of  $\mu$ whose one side contains $u_i$, for some arbitrary but fixed $1\leq i \leq k$. Denote by $w_1$ the center
of the side of $r_1$ opposite the side containing $u_i$. Let $r_2$ be other rhombus of $\mu$ that has a side containing $w_1$. Denote by $w_2$ the center of the side of $r_2$ opposite the side containing $w_1$. Continue our rhombi selecting process by picking a new rhombus $r_3$ of $\mu$ that has a side containing $w_2$. This process gives  a path of rhombi growing upward, and ending in a rhombus containing one of the $v_j$'s (see the paths of shaded rhombi in Figure \ref{QH2}(b)). We can identify this path of rhombi with the linear path $u_i\rightarrow w_1\rightarrow w_2\rightarrow w_3\rightarrow\dotsc \rightarrow v_j$ (see the dotted paths in Figure \ref{QH2}(b)).

\begin{figure}\centering
\resizebox{!}{12.5cm}{
\begin{picture}(0,0)%
\includegraphics{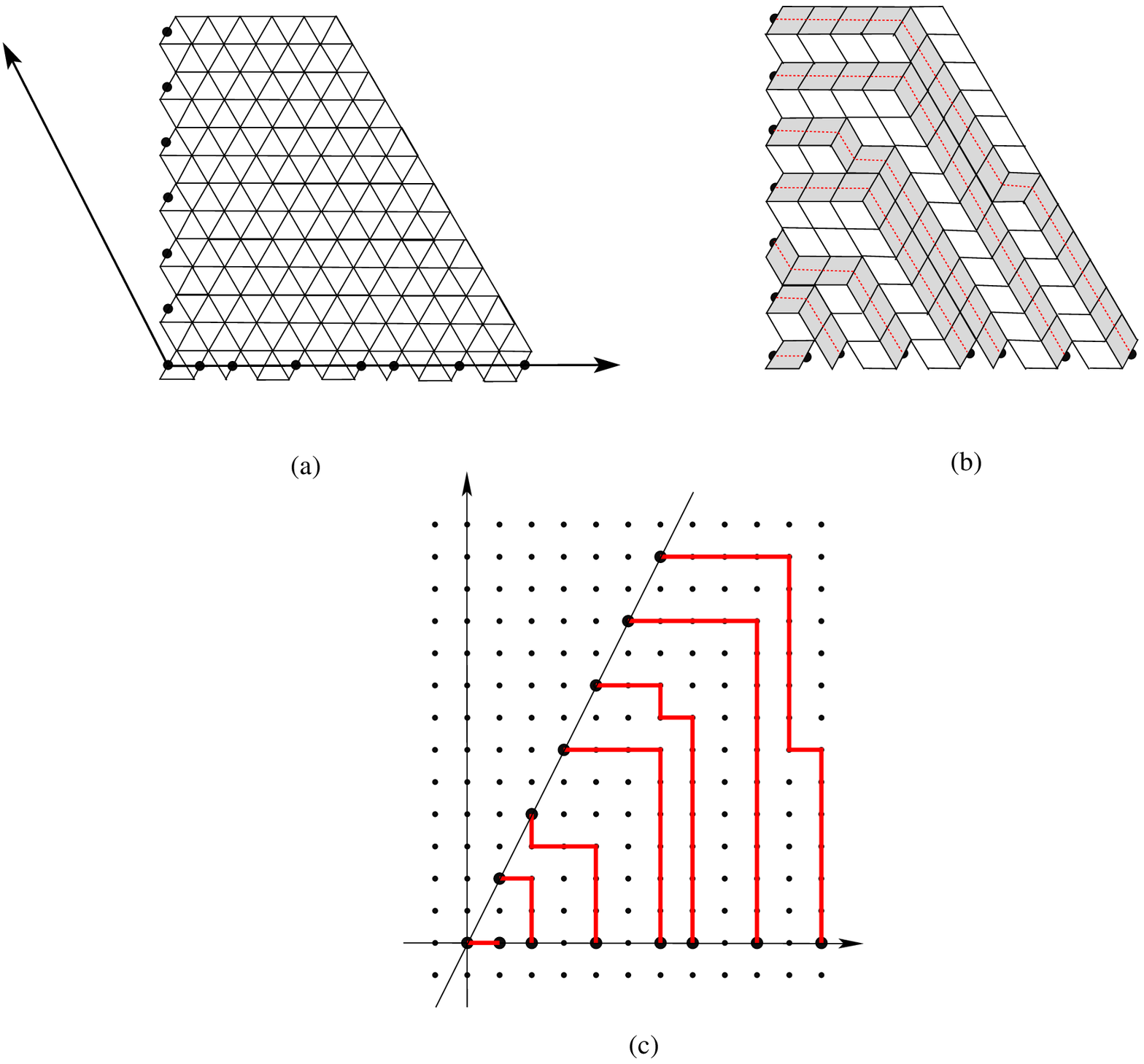}%
\end{picture}%
\setlength{\unitlength}{3947sp}%
\begingroup\makeatletter\ifx\SetFigFont\undefined%
\gdef\SetFigFont#1#2#3#4#5{%
  \reset@font\fontsize{#1}{#2pt}%
  \fontfamily{#3}\fontseries{#4}\fontshape{#5}%
  \selectfont}%
\fi\endgroup%
\begin{picture}(8469,7890)(347,-8212)
\put(1243,-3238){\makebox(0,0)[lb]{\smash{{\SetFigFont{12}{14}{\rmdefault}{\mddefault}{\updefault}{$O$}%
}}}}
\put(4640,-3418){\makebox(0,0)[lb]{\smash{{\SetFigFont{12}{14}{\rmdefault}{\mddefault}{\updefault}{$x$}%
}}}}
\put(575,-523){\makebox(0,0)[lb]{\smash{{\SetFigFont{12}{14}{\rmdefault}{\mddefault}{\updefault}{$y$}%
}}}}
\put(5554,-4134){\makebox(0,0)[lb]{\smash{{\SetFigFont{12}{14}{\rmdefault}{\mddefault}{\updefault}{$y=2x$}%
}}}}
\put(5798,-538){\makebox(0,0)[lb]{\smash{{\SetFigFont{12}{14}{\rmdefault}{\mddefault}{\updefault}{$v_7$}%
}}}}
\put(5800,-957){\makebox(0,0)[lb]{\smash{{\SetFigFont{12}{14}{\rmdefault}{\mddefault}{\updefault}{$v_6$}%
}}}}
\put(5803,-1380){\makebox(0,0)[lb]{\smash{{\SetFigFont{12}{14}{\rmdefault}{\mddefault}{\updefault}{$v_5$}%
}}}}
\put(5790,-1837){\makebox(0,0)[lb]{\smash{{\SetFigFont{12}{14}{\rmdefault}{\mddefault}{\updefault}{$v_4$}%
}}}}
\put(5775,-2264){\makebox(0,0)[lb]{\smash{{\SetFigFont{12}{14}{\rmdefault}{\mddefault}{\updefault}{$v_3$}%
}}}}
\put(5797,-2648){\makebox(0,0)[lb]{\smash{{\SetFigFont{12}{14}{\rmdefault}{\mddefault}{\updefault}{$v_2$}%
}}}}
\put(5775,-3061){\makebox(0,0)[lb]{\smash{{\SetFigFont{12}{14}{\rmdefault}{\mddefault}{\updefault}{$v_1$}%
}}}}
\put(6245,-3247){\makebox(0,0)[lb]{\smash{{\SetFigFont{12}{14}{\rmdefault}{\mddefault}{\updefault}{$u_1$}%
}}}}
\put(6500,-3235){\makebox(0,0)[lb]{\smash{{\SetFigFont{12}{14}{\rmdefault}{\mddefault}{\updefault}{$u_2$}%
}}}}
\put(6953,-3221){\makebox(0,0)[lb]{\smash{{\SetFigFont{12}{14}{\rmdefault}{\mddefault}{\updefault}{$u_3$}%
}}}}
\put(7439,-3221){\makebox(0,0)[lb]{\smash{{\SetFigFont{12}{14}{\rmdefault}{\mddefault}{\updefault}{$u_4$}%
}}}}
\put(7678,-3214){\makebox(0,0)[lb]{\smash{{\SetFigFont{12}{14}{\rmdefault}{\mddefault}{\updefault}{$u_5$}%
}}}}
\put(8149,-3214){\makebox(0,0)[lb]{\smash{{\SetFigFont{12}{14}{\rmdefault}{\mddefault}{\updefault}{$u_6$}%
}}}}
\put(8678,-3214){\makebox(0,0)[lb]{\smash{{\SetFigFont{12}{14}{\rmdefault}{\mddefault}{\updefault}{$u_7$}%
}}}}
\end{picture}}
\caption{Bijection between tilings of $R_1$ and families of non-intersecting paths.}
\label{QH2}
\end{figure}

Consider next the obtuse ($120^0$ angle) coordinate system whose origin at $v_1$ and whose $x$-axis contains all the points $u_i$'s (see Figure \ref{QH}(a)). The linear path connecting $u_i$ and $v_j$ is a lattice path in this coordinate. Normalize this coordinate system and rotating it in standard position, we get a lattice path on square grid $\mathbb{Z}^2$ (see Figure \ref{QH}(c)). It is easy to see that $v_j$ has coordinate $(j-1,2j-2)$ and $u_i$ has coordinate $(a_i-1,0)$, for any $1\leq i,j\leq k$ .

We obtain this way a $k$-tuple $\mathcal{P}$ of lattice paths  in $\mathbb{Z}^2$ using north and west steps, and they cannot touch each other (since the corresponding paths of rhombi are disjoint). One readily sees that the correspondence $\mu\mapsto \mathcal{P}$ is a bijection between the set of tilings of $R_1$ and the set of $k$-tuples $\mathcal{P}$ of non-intersecting lattice paths starting at $u_1,\dotsc,u_k$, and ending at $v_1,\dotsc,v_k$. 

By Lindstr\"{o}m-Gessel-Viennot theorem (see \cite{Lind}, Lemma 1; or  \cite{Stem}, Theorem 1.2), the number of such $k$-tuples $\mathcal{P}$ of non-intersection lattice paths is given by the determinant of the $k\times k$ matrix $\mathbf{A}$ whose $(i,j)$-entry is the number of lattice paths from $u_i=(a_i-1,0)$ to $v_j=(j-1,2j-2)$ in $\mathbb{Z}^2$ ,  that is
\[\binom{a_i+j-2}{2j-2}=\frac{(a_i+j-2)!}{(2j-2)!(a_i-j)!}\]
 (assume that $\binom{a_i+j-2}{2j-2}=0$ if $a_i-j<0$). Factor out $\frac{1}{(2j-2)!}$ from the each $j$-th column of the matrix $\mathbf{A}$, for $1\leq j\leq k$, we have
\begin{equation}\label{Aeq1}
\det(\mathbf{A})=\frac{1}{0!2!\dotsc(2k-2)!} \det_{1\leq i,j\leq k} \left( (a_i-j+1)(a_i-j+2)\dotsc (a_i+j-2)\right).
\end{equation}
Swap the $j$-th and the $(n-j+1)$-th columns, for any $1\leq j\leq k$, in the matrix on the right hand side of (\ref{Aeq1}), we get a new matrix
\[\mathbf{B}=\left((a_i-n+j)(a_i-n+j+1)\dotsc (a_i+n-j-1)\right)_{1\leq i,j\leq k},\]
and
\begin{equation}\label{Aeq3}
\det(\mathbf{B})=(-1)^{n(n-1)/2} \det_{1\leq i,j\leq k} \left( (a_i-j+1)(a_i-j+2)\dotsc (a_i+j-2)\right).
\end{equation}
Apply Lemma \ref{Klem}, with $C=1$ and  $X_i=a_i$ and $A_j=n-j$, to the matrix $\textbf{B}$, we obtain
\begin{equation}\label{Aeq4}
\det (\mathbf{B})=(-1)^{n(n-1)/2}\prod_{1\leq i< j\leq k}(a_j-a_i)(a_i+a_j-1).
\end{equation}
By (\ref{Aeq1}), (\ref{Aeq3}), and (\ref{Aeq4}), we have
\begin{equation}\label{Aeq2}
 \det(\mathbf{A})=\frac{1}{0!2!\dotsc(2k-2)!}\prod_{1\leq i< j\leq k}(a_j-a_i)(a_i+a_j-1),
\end{equation}
and (\ref{QHeq1}) follows.

\begin{figure}\centering
\includegraphics[width=12cm]{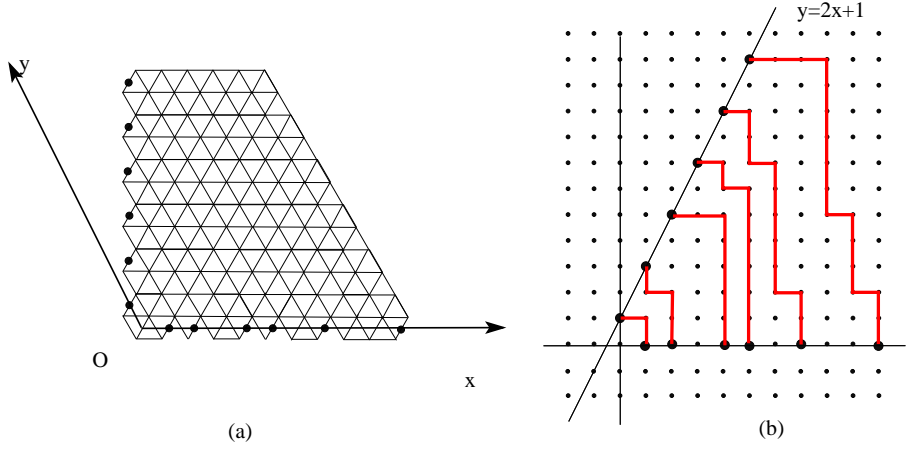}
\caption{Bijection between tilings of $R_2$ and families of non-intersecting paths.}
\label{QH3}
\end{figure}

\medskip

Next, we prove (\ref{QHeq2}) by the same method. We also have a bijection between the set of tilings of $R_2:=QH_{2k,n}(a_1,a_2,\dotsc,a_k)$ and the set of $k$-tuples of non-intersecting lattice path connecting $u_1,\dotsc,u_k$ and $v_1,\dotsc, v_k$, the only difference here is that the obtuse coordinate system is now selected so that $v_1$ has coordinate $(0,1)$ (as oppose to having coordinate $(0,0)$ in the proof of (\ref{QHeq1})). Figure \ref{QH3} illustrates an example corresponding to the region in Figure \ref{QH}(a). One readily sees  that  $u_i$ has also coordinate $(a_i-1,0)$, and $v_j$ has now coordinate $(j-1,2j-1)$ in the new coordinate system, for $1\leq i,j \leq k$. Again, by Lindstr\"{o}m-Gessel-Viennot Theorem the number of tilings of $R_2$ is given by the determinant of the $k\times k$ matrix $\mathbf{D}$ whose $(i,j)$-entry is $\binom{a_i+j-1}{2j-1}=\frac{(a_i+j-1)!}{(2j-1)!(a_i-j)!}$.

Factor out $\frac{1}{(2j-1)!}$ from the $j$-th column, and factor out $a_i$ from the $i$-th row of  the matrix $\textbf{D}$, for  any $1\leq i, j\leq k$, we get
\begin{align}\label{Deq2}
\det(\mathbf{D})=&\frac{a_1a_2\dotsc a_k}{1!3!5!\dotsc(2k-1)!}\\
&\times\det_{1\leq i,j\leq k}\left((a_i-j+1)\dotsc(a_i-1)(a_i+1)\dotsc(a_i+j-1)\right).\notag
\end{align}
Swap the $j$-th and the $(n-j+1)$-th columns, for any $1\leq j\leq k$, of the matrix on the right hand side of (\ref{Deq2}), we get a new matrix
\begin{align}
\mathbf{E}=((a_i-n+j)&(a_i-n+j+1)\dotsc(a_i-1)\notag\\
&\cdot(a_i+1)(a_i+2)\dotsc(a_i+n-j))_{1\leq i,j\leq k},
\end{align}
 and
\begin{equation}\label{Deq3}
\det(\mathbf{D})=(-1)^{n(n-1)/2}\frac{a_1a_2\dotsc a_k}{1!3!5!\dotsc(2k-1)!}\det(\mathbf{E}).
\end{equation}
Apply Lemma \ref{Klem}, with $C=0$ and  $X_i=a_i$ and $A_j=n-j$, to the matrix $\mathbf{E}$, we have
\begin{equation}\label{Deq4}
\det (\mathbf{E})=(-1)^{n(n-1)/2}\prod_{1\leq i< j\leq k}(a_j-a_i)(a_i+a_j).
\end{equation}
Therefore, by (\ref{Deq2})--(\ref{Deq4}), we obtain
\begin{equation}
\det(\mathbf{D})=\frac{a_1a_2\dotsc a_k}{1!3!5!\dotsc(2k-1)!}\prod_{1\leq i< j\leq k}(a_j-a_i)(a_i+a_j),
\end{equation}
which implies (\ref{QHeq2}).

\begin{figure}\centering
\includegraphics[width=14cm]{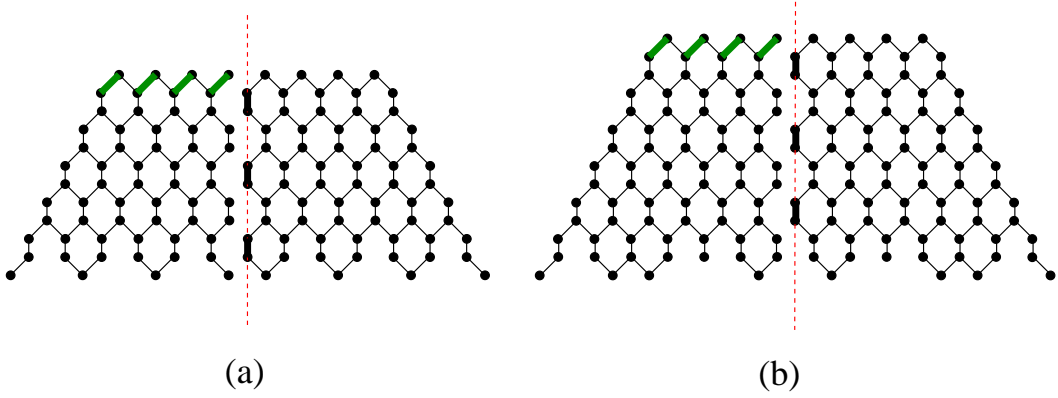}
\caption{Illustrating the proof of Theorem \ref{QHex}.}
\label{factorQH}
\end{figure}

\medskip

Apply the Factorization Theorem to the dual graph $G$ of the semi-hexagon $SH_{2k,2n}(S)$, where the triangles removed from the bottom are at the positions in the set $S:=\{n+1-a_k,a+1-a_{k-1},\dots,n+1-a_1\}\cup\{n+a_1,n+a_2,\dots,n+a_k\}$.  We get $G^-$ is isomorphic to the dual graph of the region $\overline{QH}_{2k,n}(a_1,\dotsc,a_k)$; and after removing all forced edges on the top of $G^+$, we get a graph isomorphic to the dual graph of $QH_{2k-1,n}(a_1,\dotsc,a_k)$ (see Figure \ref{factorQH}(a) for an example with $k=3$, $n=7$, $a_1=2$, $a_2=4$, $a_3=6$). Therefore, we obtain
\begin{equation}\label{QHfactor1}
\M(SH_{2k,2n}(S))=2^{k}\M(QH_{2k-1,n}(a_1,\dots,a_k))\M(\overline{QH}_{2k,n}(a_1,\dots,a_k)).
\end{equation}
Similarly, apply the Factorization Theorem to the dual graph of the semi-hexagon\\ $SH_{2k+1,2n+1}(S')$, where $S':=\{n+1-a_k,a+1-a_{k-1},\dots,n+1-a_1\}\cup\{n+1\}\cup\{n+1+a_1,n+1+a_2,\dots,n+1+a_k\}$  (see Figure \ref{factorQH}(b) for an example with $k=3$, $n=7$, $a_1=2$, $a_2=3$, $a_3=6$), we get
\begin{equation}\label{QHfactor2}
\M(SH_{2k+1,2n+1}(S'))=2^{k}\M(QH_{2k,n}(a_1,\dots,a_k))\M(\overline{QH}_{2k+1,n}(a_1,\dots,a_k)).
\end{equation}
We define an operation $\Delta$ by setting
 \[\Delta(A)=\prod_{1\leq i<j \leq k} (s_j-s_i),\]
for any finite set $A:=\{s_1,s_2,\dotsc,s_k\}$.  One can check that
\[\Delta(S)=\left(\prod_{1\leq i<j\leq k}(a_j-a_i)\right)^2\prod_{1\leq i,j\leq k}(a_i+a_j+1)\]
and
\[\Delta(S')=\left(\prod_{1\leq i<j\leq k}(a_j-a_i)\right)^2\left(\prod_{i=1}^{k}a_i\right)^2\prod_{1\leq i,j\leq k}(a_i+a_j),\]
for the above sets $S$ and $S'$.

Thus, by (\ref{QHeq1}) and (\ref{QHfactor1}), together with Lemma \ref{semihex}, we have
\begin{align}
\M(&\overline{QH}_{2k,n}(a_1,\dots,a_k))=\frac{2^{-k}\Delta(S)}{0!1!2!3!\dotsc(2k-1)!\cdot \M(QH_{2k-1,n}(a_1,\dots,a_k))}\\
&=\frac{2^{-k}\Delta(S)}{0!1!2!3!\dotsc(2k-1)!\cdot2^{-k^2}\E(a_1,\dotsc,a_k)}\\
&=\frac{2^{-k}}{1!3!\dotsc(2k-1)!}\prod_{1\leq i< j\leq k}(a_j-a_i)\prod_{1\leq i\leq j\leq k}(a_i+a_j-1),
\end{align}
which completes the proof of (\ref{QHeq4}).

Analogously, by the Lemma \ref{semihex}, (\ref{QHeq2}) and (\ref{QHfactor2}), we obtain (\ref{QHeq3}).
\end{proof}

\begin{remark}
If we construct the nonintersecting lattice paths in the region $QH_{2k-1,n}(a_1,a_2,$ $\dotsc,a_k)$ in different way: by starting from the rhombi on the north side to the rhombi on the south side, we will get different family of non-intersecting lattice paths (see Figure 3.5(a) in \cite{Krat}). These new families of lattice paths were enumerated in \cite{Krat} under the name \textit{stars}. However, the number of stars obtained in \cite{Krat} has a different form from the one in Theorem 1.1, and the authors of \cite{Krat} did not consider stars in a bijection with the tilings of $QH_{2k-1,n}(a_1,\dotsc,a_k)$ or any other regions.
\end{remark}

\section{Proof of Theorems \ref{main} and \ref{mainv}}
The dual graph of an Aztec rectangle is called an \textit{Aztec rectangle graph}, denoted by $AR_{m,n}$  the Aztec rectangle graph of order $(m,n)$ (see Figure \ref{ARregion}(c) and \ref{HoleyAR}(a) for examples).

Before presenting the proof of Theorems \ref{main} and \ref{mainv}, we quote two results about the number of perfect matchings of  an Aztec rectangle  graph with holes (i.e, vertices removed) on the bottom.
\begin{figure}\centering
\includegraphics[width=12cm]{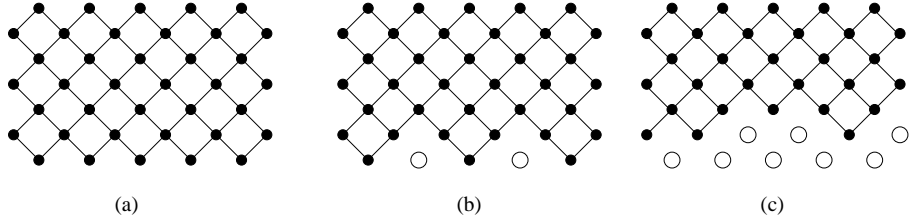}
\caption{Aztec rectangle and two holey Aztec rectangles of order $3\times 5$. The white circles indicate the removed vertices.}
\label{HoleyAR}
\end{figure}

\begin{lemma}[see Theorem 2 in cite{Mills}] \label{lem4}The number of perfect matchings of a $m\times n$ Aztec rectangle, where all the vertices in the bottom-most row, except for the $a_1$-st, the $a_2$-nd, $\dots$, and the $a_m$-th vertex, have been removed (see Figure \ref{HoleyAR}(b) for an example with $m=3$, $n=5$, $a_1=1$, $a_2=3$, $a_3=5$), equals
\begin{equation}
2^{m(m+1)/2}\prod_{1\leq i<j\leq m}\frac{a_j-a_i}{j-i}.
\end{equation}
\end{lemma}

Next, we consider a variant of the lemma above (see \cite{Gessel}, Lemma 2).

\begin{lemma}\label{lem5} The number of perfect matchings of a $m\times n$ Aztec rectangle, where all the vertices in the bottom-most row have been removed, and where the $a_1$-st, the $a_2$-nd, $\dots$, and the $a_m$-th vertex, have been removed from the resulting graph (see Figure \ref{HoleyAR}(c),  for and example with $m=3$, $n=5$, $a_1=3$, $a_2=4$,$a_3=6$), equals
\begin{equation}
2^{m(m-1)/2}\prod_{1\leq i<j\leq m}\frac{a_j-a_i}{j-i}.
\end{equation}
\end{lemma}

Denote by $AR_{m,n}(a_1,\dotsc,a_m)$ and $\overline{AR}_{m,n}(a_1,\dotsc,a_m)$ the graphs in Lemmas \ref{lem4} and \ref{lem5}, respectively.

\begin{proof}[Proof of Theorem \ref{main}]

By Theorems \ref{QHex} (equality (\ref{QHeq1})), Lemma \ref{lem2} (equality (\ref{eq8})), and Lemma \ref{QHex2} (equality (\ref{eq3})), we get (\ref{main4}). From (\ref{main4}), Lemmas \ref{lem1} (equality (\ref{eq5})) and  Lemma \ref{lem2} (equality(\ref{eq1})), we deduce (\ref{main1}).

\begin{figure}\centering
\includegraphics[width=8cm]{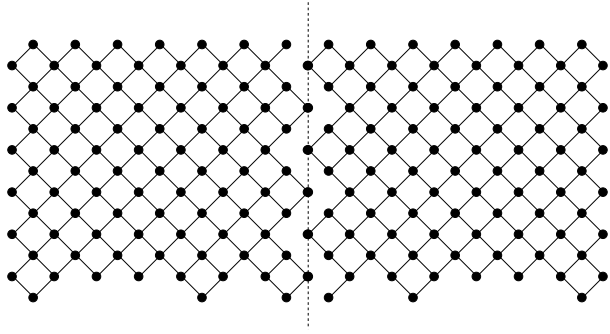}
\caption{Illustrating the proof of Theorem \ref{main}}
\label{factor2}
\end{figure}

Apply the Factorization Theorem to the graph $AR_{2k,2n}(S)$, where $S=\{n+1-a_k, n+1-a_{k-1},\dotsc,n+1-a_1\}\cup\{n+a_1,n+a_2,\dotsc,n+a_k\}$ (see Figure \ref{factor2} for an example with $n=7$, $k=3$, $a_1=1$, $a_2=3$, $a_3=7$), we get
\begin{equation}\label{factoreq2}
\M(AR_{2k,2n}(S))=2^{k}\M(RE_{2k,n}(a_1,a_2,\dotsc,a_k))\M(RO_{2k,n}(a_1,a_2,\dotsc,a_k)).
\end{equation}
Similar to the proof of the equality (\ref{QHeq4}) in Theorem \ref{QHex}, by equalities  (\ref{main1}),  (\ref{factoreq2}) and Lemma \ref{lem4}, we obtain
\begin{align}\label{factoreq3}
\M(RO_{2k,n}&(a_1,a_2,\dotsc,a_k))=\frac{2^{k(2k+1)}\Delta(S)}{0!1!2!\dotsc(2k-1)! \cdot 2^{k} \Od(a_1,\dotsc a_k)}\\
&=\frac{2^{k^2}}{1!3!\dotsc(2k-1)!}\prod_{1\leq i< j\leq k}(a_j-a_i)\prod_{1\leq j\leq i\leq k}(a_i+a_j-1).
\end{align}
Thus, Lemma \ref{lem2} implies (\ref{main2}).

Finally, by Lemma \ref{lem1} (equality (\ref{eq7})), Lemma \ref{lem2} (equality (\ref{eq2})), and the equality (\ref{main2}), we get (\ref{main3}).
\end{proof}

\medskip

\begin{proof}[Proof of Theorems \ref{mainv}]
By considering forced edges, we have the following facts similar to Lemma \ref{lem2}:
\begin{equation}\label{fn1}
\M(\overline{RO}_{2k-1,n}(a_1,a_2,\dotsc,a_k))=\M(\overline{RO}_{2k,n}(a_1,a_2,\dotsc,a_k)),
\end{equation}
\begin{equation}\label{fn2}
\M(\overline{RE}_{2k,n}(a_1,a_2,\dotsc,a_k))=\M(\overline{RE}_{2k+1,n}(a_1,a_2,\dotsc,a_k)),
\end{equation}
\begin{equation}\label{fn3}
\M(\overline{TO}_{2k+1,n}(a_1,a_2,\dotsc,a_k))=\M(\overline{TO}_{2k+2,n}(a_1,a_2,\dotsc,a_k)),
\end{equation}
\begin{equation}\label{fn4}
\M(\overline{TE}_{2k,n}(a_1,a_2,\dotsc,a_k))=\M(\overline{TE}_{2k+1,n}(a_1,a_2,\dotsc,a_k)).
\end{equation}
By using the four fundamental Lemmas \ref{VS}, \ref{star}, \ref{spider} and \ref{4cycle} as in the proof Lemma \ref{lem1}, one can get
\begin{equation}\label{QHex4}
\M(\overline{RE}_{2k,n}(a_1,a_2,\dotsc,a_k))=2^{k}\M(\overline{TE}_{2k,n}(a_1,a_2,\dotsc,a_k)).
\end{equation}

We get (\ref{VReq3}) from the equality (\ref{fn4}), Lemma \ref{QHex2} (the equality (\ref{QHex3})), and Lemma\ref{QHex} (the equality (\ref{QHeq2})). Moreover, by (\ref{QHex4}), (\ref{fn2}) and (\ref{VReq4}), we obtain (\ref{VReq1}).

Factorization Theorem implies
\begin{equation}\label{factorn2}
\M(\overline{AR}_{2k+1,2n}(S'))=2^k\M(\overline{TO}_{2k+1,n}(a_1,\dotsc,a_k))\M(\overline{TE}_{2k,n}(a_1,\dotsc,a_k))
\end{equation}
and
\begin{equation}\label{factorn1}
\M(AR_{2k,2n-1}(S''))=2^k\M(\overline{RO}_{2k-1,n}(a_1,\dotsc,a_k))\M(\overline{RE}_{2k,n}(a_1,\dotsc,a_k)),
\end{equation}
where $S':=\{n+1-a_k,n+1-a_{k-1},\dots,n+1-a_1\}\cup\{n+1\}\cup\{n+1+a_1,n+1+a_2,\dots,n+1+a_k\}$, and where
$S'':=\{n-a_k,n-a_{k-1},\dots,n-a_1\}\cup\{n\}\cup\{n+a_1,n+a_2,\dots,n+a_k\}$.
It is easy to verify that \[\Delta(S')=\Delta(S)=\left(\prod_{i=1}^{k}a_i\right)^2\left(\prod_{1\leq i<j\leq k}(aj-a_i)\right)^2\prod_{1\leq i,j \leq k
}(a_i+a_j).\]
Therefore, by (\ref{fn3}) and (\ref{factorn2}), we  get
\begin{align}
\M(\overline{TO}_{2k+2,n}(a_1,\dotsc,a_k))&=\M(\overline{TO}_{2k+1,n}(a_1,\dotsc,a_k))=\frac{\M(\overline{AR}_{2k+1,2n}(S'))}{2^k\M(\overline{TE}_{2k,n}(a_1,\dotsc,a_k))}\\
&=\frac{2^{(2k+1)k}\Delta(S')}{0!1!2!3!\dotsc(2k)!2^{k}\overline{\Od}(a_1,\dotsc,a_k)}\\
&=\frac{2^{k^2}\prod_{i=1}^{k}a_i}{0!2!4!\dotsc(2k)!}\prod_{1\leq i<j\leq k}(a_j-a_i)\prod_{1\leq i\leq j\leq k}(a_i+a_j),
\end{align}
which deduces (\ref{VReq4}).

Similarly, by (\ref{fn1}) and (\ref{factorn1}), we  obtain
\begin{align}
\M(\overline{RO}_{2k,n}(a_1,\dotsc,a_k))&=\M(\overline{RO}_{2k-1,n}(a_1,\dotsc,a_k))=\frac{\M(AR_{2k,2n-1}(S''))}{2^k\M(\overline{RE}_{2k,n}(a_1,\dotsc,a_k))}\\
&=\frac{2^{(2k+1)k}\Delta(S')}{0!1!2!3!\dotsc(2k-1)!2^{2k}\overline{\Od}(a_1,\dotsc,a_k)}\\
&=\frac{2^{k(k-1)}\prod_{i=1}^{k}a_i}{0!2!4!\dotsc(2k-2)!}\prod_{1\leq i<j\leq k}(a_j-a_i)\prod_{1\leq i\leq j\leq k}(a_i+a_j).
\end{align}
Then (\ref{VReq2}) follows.
\end{proof}

\end{document}